\documentclass[11pt]{article}
\usepackage [dvips]{graphics}
\usepackage[centertags]{amsmath}
\usepackage{amsfonts}
\usepackage{amssymb}
\usepackage{amsthm}
\usepackage{newlfont}
\usepackage{stmaryrd}
\usepackage{color,epsfig}
\usepackage{amsfonts,graphicx,psfrag}
\usepackage{graphicx}
\usepackage{float}

\usepackage{txfonts}
\usepackage{mathrsfs}

\theoremstyle{plain}
\newtheorem{thm}{Theorem}[section]
\newtheorem{cor}[thm]{Corollary}
\newtheorem{lem}[thm]{Lemma}

\newtheorem{exam}[thm]{Example}
\newtheorem{rem}[thm]{Remark}
\newtheorem{defi}[thm]{Definition}

\def\sqr#1#2{{\vcenter{\vbox{\hrule height.#2pt
              \hbox{\vrule width.#2pt height#1pt \kern#1pt \vrule
width.#2pt}
              \hrule height.#2pt}}}}
\def\be{\begin{equation}}
\def\ee{\end{equation}}

\def\ep{{\epsilon}}

\def\Sp{{\mathrm {Sp}}}

\def\ind{\text{ind}}

\def\D{\Delta}

\def\no{\noindent}

\def\bs{\bigskip}
\def\q{\quad}

\def\dim{\hbox{\rm dim$\,$}}

\def\sgn{\hbox{\rm sgn$\,$}}

\def\({\Big (}
\def\){\Big )}
\def\[{\Big[}
\def\]{\Big]}

\def\be{\begin{equation}}
\def\bel{\begin{equation}\label}
\def\ee{\end{equation}}
\def\bea{\begin{eqnarray}}
\def\eea{\end{eqnarray}}
\def\bt{\begin{theorem}}
\def\et{\end{theorem}}
\def\bc{\begin{corollary}}
\def\ec{\end{corollary}}
\def\bl{\begin{lemma}}
\def\el{\end{lemma}}
\def\bp{\begin{proposition}}
\def\ep{\end{proposition}}
\def\br{\begin{remark}}
\def\er{\end{remark}}
\def\ba{\begin{array}}
\def\ea{\end{array}}
\def\bd{\begin{definition}}
\def\ed{\end{definition}}

\DeclareMathOperator{\image}{Im}

\makeatletter
   
   \@addtoreset{equation}{section}
\makeatother

\begin{document}

\title{\bf  Morse Index Theorem of Lagrangian Systems and Stability of Brake Orbit}
\author
{Xijun Hu\thanks{Partially supported
by NSFC(No.11425105, 11790271), E-mail:xjhu@sdu.edu.cn }  \quad Li Wu \thanks{
 E-mail: 201790000005@sdu.edu.cn} \quad Ran Yang \thanks{
 E-mail: yangran201311260@mail.sdu.edu.cn }    \\ \\
 Department of Mathematics, Shandong University\\
Jinan, Shandong 250100, The People's Republic of China\\
}

\date{}
\maketitle
\begin{abstract}
In this paper, we prove  Morse index theorem of Lagrangian system with  any self-adjoint  boundary conditions.   Based on it, we give some nontrivial estimation on the  difference of Morse indices.    As an application,   we get a new criterion for  the stability problem of  brake periodic orbit.

 %In this paper, the difference of Morse indices of Lagrangian systems between any self-adjoint boundary condition and Dirichlet boundary condition is computed directly by using triple index. As corollary, we prove a Morse index theorem of Lagrangian system with general boundary conditions. At the end, we give an application to the stability problem of a brake orbit.
\end{abstract}

\bs

\no{\bf AMS Subject Classification:}  37B30, 53D12, 34B24, 37J25

\bs

\no{\bf Key Words}. Morse index, Maslov index,  Lagrangian systems,  self-adjoint boundary conditions,  stability of brake orbits

\section{Introduction}

 In this paper, we consider the Morse index theorem  of a Lagrangian system with general boundary conditions.  Precisely, let
  \begin{equation}
\mathcal{F}(x)=\int_{0}^{T}L(t,x,\dot{x})dt, \label{eq:L-S}
\end{equation}
where $x\in W^{1,2}([0,T],\mathbf{R}^{n})$ and $L\in C^{2}([0,T]\oplus \mathbf{R}^{2n},\mathbf{R})$ satisfying the Legendre convexity condition:
 \begin{equation}
 (D^{2}_{vv}L(t,u,v)w,w)>0  \quad  \text{for}  \quad  t\in [0,T], \ w\in \mathbf{R}^{n}\setminus\{0\},\ (u,v)\in \mathbf{R}^{n}\oplus \mathbf{R}^{n}. \nonumber
 \end{equation}
A solution $x$ of the corresponding Euler-Lagrange equation
\begin{equation}\label{eq:E-L equation}
\frac{d}{dt}\frac{\partial L}{\partial \dot{x}}(t,x,\dot{x})-\frac{\partial L}{\partial x}(t,x,\dot{x})=0
\end{equation}
will be called a stationary point. Linearization of \eqref{eq:E-L equation} along its stationary point is given by the following Sturm-Liouville system
\begin{equation}\label{eq:S-L equation}
 \mathcal{A}x(t):=-\frac{d}{dt}(P(t)\dot{x}(t)+Q(t)x(t))+Q(t)^{T}\dot{x}(t)+R(t)x(t)=0,
  \end{equation}
 where $P(t)=\frac{\partial^2 L}{\partial\dot{x}^2},Q(t)=\frac{\partial^2 L}{\partial x\partial\dot{x}}$ and $R(t)=\frac{\partial^2 L}{\partial x^2}$.  In fact, $\mathcal{A}$ can be considered as $\mathcal{F}''(x)$ which is the Hessian of $\mathcal{F}$ at $x$.

The boundary condition of \eqref{eq:E-L equation} can be given in the following way.
 Let $\langle\cdot,\cdot\rangle$ be the standard Hermitian inner product in $\mathbf{C}^{2n}$ and $J_n=\begin{bmatrix} 0&-I_{n}\\I_{n}&0 \end{bmatrix}$,
 then $(\mathbf{C}^{2n},\omega)$ can be seen as a complex symplectic vector space with the symplectic form $\omega(x,y)=\langle Jx,y\rangle$ for any
 $x,y\in \mathbf{C}^{2n}.$ A complex subspace $\Lambda$ is Lagrangian if and only if $\omega\mid _{\Lambda}=0 $ and $\dim_{\mathbf{C}}\Lambda=n$.  We denote by  $Lag(\mathbf{C}^{2n},\omega)$
  the set of Lagrangian subspaces.
 Let $y(t)=P(t)\dot{x}(t)+Q(t)x(t), \ z(t)=(y(t),x(t))^T$, then we consider the most general self-adjoint  boundary conditions, namely,
 \begin{equation}\label{eq:original general boundary condition}
 (z(0),z(T))\in\Lambda_0,
 \end{equation}
 where $\Lambda_0\in Lag(\mathbf{C}^{2n}\oplus\mathbf{C}^{2n},-\omega\oplus\omega)$. Obviously, $\mathcal{A}$ is a self-adjoint operator on $L^2([0,T], \mathbf{C}^{2n} )$ with domain $$E_{\Lambda_0}(0,T):=\{x\in W^{2,2}([0,T], \mathbf{C}^{2n}), (z(0),z(T))\in\Lambda_0)   \}. $$
For a self-adjoint operator $A$, we denote its Morse index by $m^-(A)$, which is  the number of total negative eigenvalues of $A$.  Throughout of the paper, we always denote $m^0(A)=\dim\ker(A)$ and $m^+(A)$ be the total number of positive eigenvalues of $A$.   For a critical point $x$ of $\mathcal{F}$, we define the  Morse index by
$$ m^-(x, \Lambda_0)=m^-(\mathcal{A}), $$
and we always omit $x$ when there is no confusion.

Obviously, the Morse index depends on the boundary condition $\Lambda_0$. Let $\Lambda^n_N=\{0\}\oplus \mathbf{C}^n$, $\Lambda^n_D=\mathbf{C}^n\oplus\{0\}$, which is Lagrangian subspace of $(\mathbf{C}^{2n},\omega)$ and can be   considered as  the Neumann and Dirichlet boundary conditions.  For convenience, we set
$\Lambda_D=\Lambda^n_D\oplus\Lambda^n_D$ and  $\Lambda_N=\Lambda^n_N\oplus\Lambda^n_N$.  Obviously,
$$\Lambda_N=\{(z(0),z(T))\in\mathbf{C}^{4n}, y(0)=y(T)=0\}, \quad  \Lambda_D=\{(z(0),z(T))\in\mathbf{C}^{4n}, x(0)=x(T)=0 \}, $$
which means both start time and end time with Neumann or Dirichlet boundary conditions.

It is well known that the Morse index with Dirichlet boundary condition can be expressed as sum of conjugate points, please refer \cite{HS} and reference therein for the detail.  More precisely, by the standard Legendre transformation, the system (\ref{eq:S-L equation}) becomes into
\begin{eqnarray}
\dot{z}(t)=J\mathcal{B}(t)z(t), \label{eq:Hamilton system}
\end{eqnarray}
where
\begin{equation}
\label{eq:the form of B}\mathcal{B}(t)=\begin{bmatrix} P^{-1}(t) &- P^{-1}(t)Q(t)\\-Q^T(t)P^{-1}(t)&Q^{T}(t)P^{-1}(t)Q(t)-R(t) \end{bmatrix}.
\end{equation}
Let $\gamma(t)$ be the fundamental solution of  \eqref{eq:Hamilton system}, that is $\gamma(0)=I_{2n}$ and $\dot{\gamma}(t)=J\mathcal{B}(t)\gamma(t)$.  It is well known that
$$ \gamma(t)\in \Sp(2n):=\{M^TJM=J, M\in GL(\mathbf{R}^{2n})  \}. $$ We have the well known Morse index Theorem
\bea m^-(\Lambda_D)=\sum_{0<t_0<T} \dim\ker \gamma(t_0)\Lambda_D\cap\Lambda_D. \label{mor1} \eea
For general boundary condition $\Lambda_0$, we can't get a simple formula as in \eqref{mor1}.

In order to give the  Morse index Theorem with general boundary conditions, we  consider the difference  of $m^-(\Lambda_0)-m^-(\Lambda_D)$. The difference is expressed by Duistermaat triple index \cite{D}.
 Let $\alpha, \beta$ and $\delta$ be three isotropic subspaces of complex symplectic vector space $(V,\omega)$, the triple index $i(\alpha,\beta,\delta)$ is well defined  \cite{ZWZ} and satisfied
 \bea 0\leq i(\alpha,\beta,\delta)\leq \dim\alpha-\dim(\alpha\cap\beta+\beta\cap\delta). \label{triple inequality} \eea
For a matrix $M\in L(\mathbf{C}^m)$, we always denote by  $$Gr(M):=\{(x,Mx), x\in\mathbf{C}^m \},$$
which is a linear subspace of $\mathbf{C}^m\oplus \mathbf{C}^m$.  Since $\gamma(T)\in\Sp(2n)$, it is obvious that $Gr(\gamma(T))\in Lag(\mathbf{C}^{2n}\oplus\mathbf{C}^{2n},-\omega\oplus\omega)$, and  $i(Gr(\gamma(T)), \Lambda_0, \Lambda_D)$ is well defined.
\begin{thm}\label{thm:difference of two morse index}
For a critical point $x$ of the Lagrangian system \eqref{eq:L-S}, the Morse indices $m^{-}(\Lambda_0)$ and $m^-(\Lambda_D)$ satisfy
 \begin{equation}
 m^{-}(\Lambda_0)- m^-(\Lambda_D)=i(Gr(\gamma(T)), \Lambda_0, \Lambda_D).
 \end{equation}
\end{thm}

 Let $V(\Lambda_0)$ be the subspace of $\Lambda_N$ defined by $$V(\Lambda_0)=(\Lambda_0+\Lambda_D)\cap\Lambda_N,$$  then $(x(0),x(T))^T\in V(\Lambda_0)$.
We always denote $$\nu(\Lambda_0)=\dim V(\Lambda_0).$$
Obviously,  $\nu(\Lambda_D)=0$ and $\nu(\Lambda_N)=2n$.  Direct computations show that for the periodic boundary condition $\Lambda_P$, $\nu(\Lambda_P)=n$.

Please note that \eqref{tri1} implies \bea i(Gr(\gamma(T)), \Lambda_0, \Lambda_D)\leq \nu(\Lambda_0), \eea we have  the following inequality.

\begin{cor}\label{cor1.1}
\begin{equation}\label{eq:inequality of morse index}
m^-(\Lambda_D)\leq m^-(\Lambda_0)\leq m^-(\Lambda_D)+\nu(\Lambda_0). \end{equation}
\end{cor}

It is well known that for Dirichlet boundary condition the Morse index has some monotone property \cite{CH}. To be precisely, for the linear system \eqref{eq:S-L equation} defined on  time interval $[a,b]$.   Let $$m^-_{[a,b]}(\Lambda_0)=m^-(\mathcal{A}|_{E_{\Lambda_0}(a,b)}).$$   For $[c,d]\subset[a,b]$,  we have
\bea m^-_{[a,b]}(\Lambda_D)\geq  m^-_{[a,c]}(\Lambda_D)+m^-_{[c,d]}(\Lambda_D)+m^-_{[d,b]}(\Lambda_D).  \eea

This property is not true for general boundary conditions, instead, as a corollary of \eqref{eq:inequality of morse index}, we have the next estimation.
\begin{cor}\label{cor1.1}
\begin{equation}\label{eq:inequality of morse index1}
m_{[a,b]}^-(\Lambda_0)\geq m_{[c,d]}^-(\Lambda_0)+m_{[a,c]}^-(\Lambda_D)+m_{[d,b]}^-(\Lambda_D)-\nu(\Lambda_0).
\end{equation}
\end{cor}

Let $\Lambda_s\in Lag(\mathbf{C}^n, -\omega)$ and $\Lambda_e\in Lag(\mathbf{C}^n, \omega)$, then the separated boundary conditions can be given by $\Lambda_0=\Lambda_s\oplus\Lambda_e$.  The $\Lambda_s$ conjugate point is defined by $\gamma(t)\Lambda_s\cap\Lambda^n_D\neq \{0\}$. We give the general Morse index Theorem as follows
\begin{thm}\label{thm: morse index thm gene}

For a critical point $x$ of the Lagrangian system \eqref{eq:L-S} with the separated boundary condition, the Morse index $m^{-}(\Lambda_s\oplus\Lambda_e)$ satisfies
\bea
 m^-(\Lambda_s\oplus\Lambda_e)=\sum_{0<t_0<T} \dim(\gamma(t_0)\Lambda_s\cap\Lambda^n_D)+ i(\gamma(T)\Lambda_s,\Lambda_e,\Lambda^n_D).
 \eea\label{mor2}

\end{thm}

%\begin{thm}\label{thm: morse index thm gene}
%For a critical point $x$ of the Lagrangian system \eqref{eq:L-S}, the Morse indices $m^{-}(\Lambda_0)$ and $m^-(\Lambda_D)$ satisfy
%\bea m^-(\Lambda_0)=\sum_{0<t_0<T} \dim\ker \gamma(t_0)\Lambda_s\cap\Lambda_D+  ?. \label{mor2} \eea
%\end{thm}

\begin{rem}
Theorem  \ref{thm: morse index thm gene} can be used to study the Morse index of geodesics on  Riemannian manifold.   The classical Morse index theorem for a Riemannian manifold $(M,g)$ can be traced back to \cite{M.M1}.
 The generalizations of this result  are referred to \cite{Am}, \cite{BTZ}, \cite{E}, \cite{KD}, \cite{WJ} and reference therein.   Kalish \cite{KD} proved the Morse index theorem where the end points are in submanifolds.  As a corollary of Theorem \ref{thm: morse index thm gene}, we generalize Kalish's results to the degenerate case. Please refer to Example \ref{exam:proof of kalish's result}  for the details.
\end{rem}

The Morse index under the general boundary condition  can be expressed by the Maslov index. This is first studied by Duistermaat \cite{D}. In the periodic boundary condition it is proved by \cite{Lon4}, \cite{LA}, \cite{A1}. In \cite{HS}, Hu and Sun give a clear form for the case of the boundary condition given by $(x(0),x(T))\in V\subset \Lambda_N$.  More related  results can be found in \cite{Z} and the reference therein.

For a solution $z$ of \eqref{eq:Hamilton system}, its Maslov index (refer to  \cite{Ar},  \cite{CLM}, \cite{HS}) is given by
\begin{equation}
\mu^-(z)=\mu(\Lambda_0,Gr(\gamma(t)),t\in[0,T]).
 \end{equation}
As a corollary of Theorem \ref{thm:difference of two morse index}, we can get the relationship between the Maslov index and Morse index for the boundary condition \eqref{eq:original general boundary condition}.  \begin{cor}\label{cor:relationship between M-M index}
For a critical point $x$ of the Lagrangian system \eqref{eq:L-S} with the boundary condition \eqref{eq:original general boundary condition}, its Morse index $m^{-}(x)$ and Maslov index $\mu^-(z)$ satisfy
 \begin{equation}\label{eq:relationship between M-M index}
 \mu^{-}(z)- m^-(x)=n-i(Gr(I_{2n}), \Lambda_0, \Lambda_D).
 \end{equation}
\end{cor}

Since $0\leq i(Gr(I_{2n}), \Lambda_0, \Lambda_D)\leq 2n$ for any $\Lambda_0$, we have
\begin{equation}\label{eq:relationship between M-M index2}
 -n\leq\mu^{-}(z)- m^-(x)\leq n.
 \end{equation}

%Since $0\leq i(Gr(I_{2n}), \Lambda_0, \Lambda_D)\leq n$ for any $\Lambda_0$, we have
%\begin{equation}\label{eq:relationship between M-M index2}
% \mu^{-}(z)- m^-(x)\geq0.
 %\end{equation}

 \begin{exam}\label{exm:examples}
Denote the difference between the Maslov index and Morse index by $\Delta$. By Corollary \ref{cor:relationship between M-M index}, we list several common examples to compute associated $\Delta$.

 1. $\mathbf{\cite[Theorem 1.2]{HS}}$ Let $V$ be any subspace of $\Lambda_N$ and the boundary condition is given by $\Lambda_V=\mathcal{J}V^{\perp}\oplus V$, where $\mathcal{J}=-J\oplus J$. Then $$i(Gr(I_{2n}), \Lambda_V, \Lambda_D)=n-\dim(V^\perp\cap Gr(-I_n)),\quad  \Delta(\Lambda_V)=\dim(V^\perp\cap Gr(-I_n)).$$
 Follows  \cite{HS}, we consider two important cases. The first case is given by $V=Gr(M)$, where   $M\in GL(\mathbf{C}^n)$. For this case,  $\dim(V^\perp\cap Gr(-I_n))=\dim(M-I_{n})$. Then $$%i(Gr(I_{2n}), \Lambda_V, \Lambda_D)=n-\dim(M-I_{2n}),   \quad
 \Delta(\Lambda_V)=\dim(M-I_{n}).$$
The second case is given by  $V=V_1\oplus V_2$, where
 $V_1, V_2$ be two subspaces of $\mathbf{C}^n$.  For this case $\dim(V^\perp\cap Gr(-I_n))=\dim(V^\perp_1\cap V^\perp_2)$. Then $$%i(Gr(I_{2n}), \Lambda_V, \Lambda_D)=n-\dim(V_1\cap V_2), \quad
 \Delta(\Lambda_V)=\dim(V^\perp_1\cap V^\perp_2).$$
As special cases, we have

  $\mathbf{Dirichlet \ boundary condition}$ $i(Gr(I_{2n}), \Lambda_D, \Lambda_D)=0,\quad  \Delta(\Lambda_D)=n$.

 $\mathbf{Neumann \ boundary condition}$ $i(Gr(I_{2n}), \Lambda_N, \Lambda_D)=n,\quad  \Delta(\Lambda_N)=0$.

 $\mathbf{Periodic\ boundary condition}$ $i(Gr(I_{2n}), \Lambda_P, \Lambda_D)=0,\quad  \Delta(\Lambda_P)=n$.

2. $\mathbf{Separated \ boundary \ condition}$

 For this case, we have
 \begin{equation}\label{eq:difference of morse and maslov for separated BC}
 i(Gr(I_{2n}),\Lambda_s\oplus\Lambda_e,\Lambda_D)=n-\dim(\Lambda_s\cap\Lambda^n_D)+i(\Lambda_s,\Lambda_e,\Lambda^n_D)
 \end{equation}
 and
 \begin{equation}
 \Delta(\Lambda_s\oplus\Lambda_e)=\dim(\Lambda_s\cap\Lambda^n_D)-i(\Lambda_s,\Lambda_e,\Lambda^n_D).
 \end{equation}
Particularly, for the special case where $y(0)=A_s x(0)$, $y(T)=A_e x(T)$, we have
\begin{equation}\label{eq:formula for special separated boundary condition}
i(Gr(I_{2n}),\Lambda_s\oplus\Lambda_e,\Lambda_D)=n+m^+(A_e-A_s), \quad \Delta(\Lambda_s\oplus\Lambda_e)=-m^+(A_e-A_s).
\end{equation}
Note that for this special kind of separated boundary condition, we always have $\mu^-(z)\leq m^-(\Lambda_s\oplus\Lambda_e)$. Moreover, if we take $A_s=0, A_e=I_n$, then $\mu^-(z)- m^-(\Lambda_s\oplus\Lambda_e)=-n$.
\end{exam}

%2. $\mathbf{Mixed-type \ boundary \ condition} \ \Lambda_m $
 %{\color{red}Consider the boundary condition given by
%\begin{equation}
% B_1x(0)-B_2\dot{x}(0)=0, \quad C_1x(T)-C_2\dot{x}(T)=0,
% \end{equation}
%where $B_1, B_2, C_1, C_2$ be four $n\oplus n$ complex matrices such that
 %\begin{equation}
 %rank[B_1 \ B_2]=rank[C_1 \ C_2]=n, \quad B_1B_2^*=B_2B_1^*, \quad C_1C_2^*=C_2C_1^*
 %\end{equation}
%and $\Lambda_m=\{(-Xu,Yv,B_2^*u,C_2^*v)^T | u,v \in \mathbf{C}^n\}$ is Lagrangian with $X=P(0)B_1^*+Q(0)B_2^* , Y=P(T)C_1^*+Q(T)C_2^*$. Then we have $B_2X, C_2Y$ are Hermitian.} Denote $M=\begin{bmatrix}-B_2X&0\\0&C_2Y\end{bmatrix}, \Omega_1=\{(u,v)^T\ |\ B_2^*u=C_2^*v\}$ and $\Omega_2=\{(u,v)^T\ |\ B_2^*u=C_2^*v=0, Xu=Yv\}$. Then  $i(Gr(I_{2n}), \Lambda_m, \Lambda_D)=n+m^+(M|_{\Omega_1})-\dim\Omega_2,  \Delta(\Lambda_m)=-m^+(M|_{\Omega_1})+\dim\Omega_2$.

%Particularly, if $B_2=C_2=I_{n}$, we have $\Omega_1=\{(u,u)^T\ |\ u\in\mathbf{C}^n\}, \Omega_2=\{0\}$. Then $$i(Gr(I_{2n}), \Lambda_m, \Lambda_D)=n+m^+(Y-X),  \quad \Delta(\Lambda_m)=-m^+(Y-X).$$
%\end{exam}
%\begin{rem}\label{rem:comparition to mixed-type boundary condition and Kalish result}
Recently, the authors studied the relationship between Maslov index and Morse index for Schr\"{o}dinger operators on finite interval $[0,1]$ in  \cite{HoSu} under the separated boundary condition. In above example, we give a different proof by using our general result.

%Moreover, in \cite{KD} the author proved the Morse Index Theorem in the case
%where the endpoints of a geodesic $\gamma$ are in two submanifolds $P$ and $Q$ respectively. As an application we will explain it by using our result in example \ref{exam:proof of kalish's result}.
%\end{rem}

As an application, we consider the stability problem of brake periodic orbits.  For a periodic orbit $x$ with fundamental solution $\gamma(t)$, $x$ is called spectral stable if $\sigma(\gamma(T))\subset \mathbf{U}$, is called linear stable if moreover $\gamma(T)$ is semi-simple.

For a brake periodic orbit,
 $P(t), Q(t)$ and $R(t)$ are all real $T$-periodic symmetric matrices and satisfy the following condition:
\begin{equation}\label{eq:brake property of efficient matrix}
P(-t)=P(t), \quad Q(-t)=-Q(t), \quad R(-t)=R(t), \quad \forall t\in \mathbf{R}.
\end{equation}
 As an application of Theorem \ref{thm:difference of two morse index}, we can give an estimation of the geometrical multiplicity of all non-real eigenvalues of $\gamma(T)$.
\begin{thm}\label{thm:stability of brake orbit1}
Let $\mathbf{C}^\perp=\{\lambda\in\mathbf{C}\ | \ Im(\lambda)>0\}$, for the brake orbit, we have
\begin{equation}\label{eq:estimate of eigenvalues not real1}
\dim\bigoplus_{\lambda\in\sigma(\gamma(T))\cap(\mathbf{C}^\perp\cup\{\pm1\})}\ker(\gamma(T)-\lambda I_{2n})\leq m^-(\Lambda_N)+m^0(\Lambda_N)-m^-(\Lambda_D).
\end{equation}
and
\begin{equation}\label{eq:estimate of eigenvalues not real}
\dim\bigoplus_{\lambda\in\sigma(\gamma(T))\cap\mathbf{C}^\perp}\ker(\gamma(T)-\lambda I_{2n})\leq m^-(\Lambda_N)-m^-(\Lambda_D).
\end{equation}
\end{thm}
Please note that, if  $m^-(\Lambda_N)=m^-(\Lambda_D)$, then $\sigma(\gamma(T))\subset\mathbf{R}$.
As a corollary of Theorem \ref{thm:stability of brake orbit1}, we have the next estimation.
\begin{thm}\label{thm:stability of brake orbit}
For a given brake orbit $x$ of system \eqref{eq:S-L equation}, we assume  $m^-(x,\Lambda_P)=k$, then for the monodromy matrix $\gamma(T)$ we have
\begin{equation}\label{eq:estimate of eigenvalues not real2}
\dim\bigoplus_{\lambda\in\sigma(\gamma(T))\cap\mathbf{C}^\perp}\ker(\gamma(T)-\lambda I_{2n})\leq2k.
\end{equation}
\end{thm}
In 1987, Offin proved the non-degenerate brake orbit with zero Morse index is hyperbolic \cite{Off}, some related results can be found  in \cite{BT}.  Recently,  Ure\~{n}a  proved  \cite[Theorem 1.1]{Ur} that  all the eigenvalues of
$\gamma(T)$ are not only real but also positive. Inequality \eqref{eq:estimate of eigenvalues not real2} can be considered as a generalization of their results.   In the case $m^-(x,\Lambda_P)=0$, we can get Ure\~{n}a's result by combining with other simple discussions,  the details will be given in Section $4$.

This  paper is organized as follows, in Section $2$, we briefly review the index theory for Lagrangian subspace,  and  we prove the Morse index theorem in  Section $3$, at last,we prove the stability theorem of brake orbits.

\section{Index Theory for the Lagrangian subspaces}

In this section, we briefly review the theory about Maslov index, H\"{o}rmander index, triple index and their relations.

In Hamiltonian system theory, Maslov index is an important topological characterization. About this theory, readers are referred to \cite{Ar}, \cite{D},  \cite{RS1} and so on.
 Denote $\Lambda^{\pm}=\ker(iJ\mp I_{2n})$, then $\Lambda$ is Lagrangian if and only if $\Lambda$ can be expressed as a graph of a unitary operator
 $U:\Lambda^{+}\rightarrow \Lambda^{-}$. So we can define a homeomorphic (isomorphic) map $f:Lag(\mathbf{C}^{2n},\omega)\rightarrow \mathbf{U}(n)$.
 %$ where $Lag(\mathbf{C}^{2n},\omega)$ is the set of all complex Lagrangian subspaces and $\mathbf{U}(n)$ is the unitary group.
     In fact, $f$ can be defined in the following way.
  Denote the conjugate transpose of a complex matrix $X$ by $X^*$. Recall that the Lagrangian frame of a given Lagrangian subspace $\Lambda$ is an injective linear map $\mathcal{Z}:\mathbf{C}^{n}\rightarrow \Lambda$ with the form $\mathcal{Z}=\begin{bmatrix}X\\Y \end{bmatrix}$, where $X,Y$ are $n\times n$ complex matrices such that $X^*Y=Y^*X$ and $rank(\mathcal{Z})=n$. To be convenient, we will denote Lagrangian subspace only by its Lagrangian frame later. $$ f(\mathcal{Z})=(X-iY)(X+iY)^{-1}. $$
  Note that
$$
\dim\ (\Lambda_{1}\cap \Lambda_{2})=\dim \ \ker(f(\Lambda_{2})^{-1}f(\Lambda_{1})-I_{n}).
$$
Then for any fixed $U_{0}\in \mathbf{U}(n)$, we can define the singular cycle $\Sigma_{U_{0}}$ of $U_{0}$ as
$$
 \Sigma_{U_{0}}=\{U\in \mathbf{U}(n)\mid \det(U_{0}^{-1}U-I_{n})=0\}.
$$
 Let $U_{t}(t\in [a,b])$ be any path in $\mathbf{U}(n)$. By the proper small perturbation
 $e^{is}U_{t}(|s|\leq\varepsilon)$, for any fixed $t_{0}\in [a,b]$, we have the path $e^{is}U_{t_{0}}$ is transversal to $\Sigma_{U_{0}}$ and for
  fixed small enough $s_{0}$, there holds $e^{-is_{0}}U_{a}$ and $e^{-is_{0}}U_{b}$ are all not in the singular cycle of $U_{0}$. So the intersection number
  $[e^{-is_{0}}U_{t}:\Sigma_{U_{0}}]$ is well defined. Then we can introduce the following definition:
 \begin{defi}\label{defi2.1}
Let $\Lambda(t)$ be a path in $Lag(\mathbf{C}^{2n})$ and $\Lambda_{0}\in Lag(\mathbf{C}^{2n})$, then the Maslov index is defined as
\begin{equation}
\mu(\Lambda_{0},\Lambda(t)):=[e^{-is_{0}}f(\Lambda(t)): \Sigma_{f(\Lambda_{0})}].
\end{equation}
\end{defi}

\cite{RS1} gives an effective way to compute the Maslov index $\mu(\Lambda_{0},\Lambda(t))$ by using crossing form. For the $C^1$-Lagrangian path $\{\Lambda(t), t\in[0,T]\}$, then $t_{0}$ is called a crossing if $\Lambda(t_{0})\cap \Lambda_{0}\neq{0}$. Let $v$ be any vector in $\Lambda(t_{0})\cap \Lambda_{0}$ and $V_{t_{0}}$ be a fixed Lagrangian subspace which is transversal to
$\Lambda(t_{0})$. For small $t$, the crossing form is defined by
$$
\Gamma(\Lambda(t_{0}),\Lambda_{0},t_{0})(v)=\frac{d}{dt}\mid_{t=t_{0}}\omega(v,u(t))
$$
where $u(t)\in V_{t_{0}}$ such that $v+u(t)\in\Lambda(t)$ and the form is independent of the choice of $V_{t_{0}}$. For the special Lagrangian path $\Lambda(t)=\gamma(t)W$, where $\gamma(t)\in \Sp(2n)$ and $W$ is a fixed Lagrangian subspace, then the
crossing form is $\langle-\gamma(t)^{T}J\dot{\gamma}(t)v,v\rangle$ for $v\in \gamma(t)^{-1}(\Lambda(t)\cap W)$.

A crossing is called regular if the crossing form is non-degenerate. For every $C^{1}$ path
 with fixed endpoints, we can make sure that all the crossings are regular by small perturbation. Then following $\cite{LZu}$ we have
 \begin{equation}\label{eq:maslov index formula}
\mu(\Lambda_{0},\Lambda(t))=m^{+}(\Gamma(\Lambda(0),\Lambda_{0},0))+\sum_{t\in \mathcal{S}} \sgn(\Gamma(\Lambda(t),\Lambda_{0},t))-
m^{-}(\Gamma(\Lambda(T),\Lambda_{0},T)),
\end{equation}
where $\mathcal{S}$ is the set of all crossings and $m^{+},m^{-}$ are the dimensions of positive and negative subspaces respectively.

There is another important index related to the Maslov index, namely, the H\"{o}rmander index (See \cite{RS1}). Let $V_{0},V_{1},\Lambda_{0},\Lambda_{1}$ be four Lagrangian subspaces and $\{\Lambda(t),t\in[0,T]\}$
is any Lagrangian path such that $\Lambda_{0}=\Lambda(0),\Lambda_{1}=\Lambda(T)$, then there is the following definition:
\begin{defi}\label{defi2.5}
Let $V_{0},V_{1},\Lambda_{0},\Lambda_{1}$ and $\Lambda(t)$ be as above, the H\"{o}rmander index is defined to be
\begin{equation}
s(\Lambda_{0},\Lambda_{1};V_{0},V_{1}):=\mu(V_{1},\Lambda(t))-\mu(V_{0},\Lambda(t)).
\end{equation}
\end{defi}
Please note that the symbol $s(\Lambda_{0},\Lambda_{1};V_{0},V_{1})$ here is a little different from the original one in \cite{RS1} in order to correspond to \cite[Definition 3.9]{ZWZ}.  One can check that the above definition is independent on the choice of path $\Lambda(t)$, so it's well defined. Recently, in \cite{ZWZ}, the second author and his collaborators studied the H\"{o}rmander index in the finite dimensional case by using triple index.

 Now we will introduce the definition of triple index. Generally, let $\alpha, \beta$ and $\delta$ be three isotropic subspaces of complex symplectic vector space $(V,\omega)$, then define the form $\mathfrak{Q}:=\mathfrak{Q}(\alpha, \beta; \delta)$ on $\alpha\cap(\beta+\delta)$ by
 \begin{equation}\label{eq:the form Q}
\mathfrak{Q}(x_1,x_2)=\omega(y_1, z_2),
\end{equation}
where $x_j=y_j+z_j \in \alpha\cap(\beta+\delta)$ and $y_j\in\beta, z_j\in \delta$ for $j=1,2$.
From \cite{ZWZ}, when $\alpha, \beta, \delta$ are Lagrangian subspaces, we have
\bea \ker \mathfrak{Q}(\alpha, \beta, \delta)= \alpha\cap\beta+\alpha\cap\delta.\label{Qker} \eea

Moreover, we need the following lemma \cite[Lemma 3.2]{ZWZ}:
\begin{lem}\label{lem:change order of form Q}
For three isotropic subspaces $\alpha, \beta, \delta$, let $\mathfrak{Q}_1:=\mathfrak{Q}(\alpha, \beta; \delta), \mathfrak{Q}_2:=\mathfrak{Q}(\beta, \delta; \alpha) $ and $\mathfrak{Q}_3:=\mathfrak{Q}(\delta, \alpha; \beta)$, then we have
$$
m^{\pm}(\mathfrak{Q}_1)=m^{\pm}(\mathfrak{Q}_2)=m^{\pm}(\mathfrak{Q}_3).
$$
\end{lem}

Now by equation $(2.6)$ of \cite{D}, the triple index is well-defined as following:
\begin{defi}\label{defi:triple index}
Let $\alpha, \beta$ and $\kappa$ be three Lagrangian subspaces of complex symplectic vector space $(V,\omega)$, then the triple index of $\alpha, \beta, \kappa$ is defined by
\begin{equation}\label{eq:the triple index}
i(\alpha, \beta, \kappa)=m^-(\mathfrak{Q}(\alpha, \delta; \beta))+m^-(\mathfrak{Q}(\beta, \delta; \kappa))-m^-(\mathfrak{Q}(\alpha, \delta; \kappa)),
\end{equation}
where $\delta$ is a Lagrangian subspace such that $\delta\cap\alpha=\delta\cap\beta=\delta\cap\kappa=\{0\}$.
\end{defi}
An equivalent definition can be given as follows \bea  i(\alpha, \beta, \kappa)=m^+(\mathfrak{Q}(\alpha, \beta; \kappa))+\dim(\alpha\cap\kappa)-\dim(\alpha\cap\beta\cap\kappa)),\label{trip1} \eea where $m^+$ denotes the dimension of maximal positive definite subspace which $\mathfrak{Q}$ acts on.
It follows that $$ i(\alpha,\beta,\kappa)\geq 0. $$
The triple index $i(\alpha, \beta, \kappa)$ can be calculated and estimated by \cite[Lemma 3.13]{ZWZ}:
\begin{equation}\label{eq:compute the triple index}
\begin{aligned}
i(\alpha, \beta, \kappa)&=m^+(\mathfrak{Q}(\alpha, \beta; \kappa))+\dim(\alpha\cap\kappa)-\dim(\alpha\cap\beta\cap\kappa))\\
&\leq \dim \alpha-\dim(\alpha\cap\beta)-\dim(\beta\cap\kappa)+\dim(\alpha\cap\beta\cap\kappa)\\
&=\dim \alpha -\dim(\alpha\cap\beta+\beta\cap\kappa),
\end{aligned}
\end{equation}

In particular, we have \cite[Corollary 3.14]{ZWZ}
\begin{equation}\label{eq:particular triple index}
i(\alpha, \alpha, \beta)=i(\beta, \alpha, \alpha)=0, \quad i(\alpha, \beta, \alpha)=\dim \alpha-\dim(\alpha\cap\beta).
\end{equation}
For four given Lagrangian subspaces $\lambda_1, \lambda_2, \kappa_1, \kappa_2$ of complex symplectic vector space $(V,\omega)$, the main theorem $1.1$ of \cite{ZWZ} gives an efficient way to compute the H\"{o}rmander index $s(\lambda_1, \lambda_2; \kappa_1, \kappa_2)$ by
\begin{equation}\label{eq:the Hormander index computed by triple index}
  s(\lambda_1, \lambda_2; \kappa_1, \kappa_2)=i(\lambda_1, \lambda_2, \kappa_2)-i(\lambda_1, \lambda_2, \kappa_1)
  =i(\lambda_1, \kappa_1, \kappa_2)-i(\lambda_2, \kappa_1, \kappa_2).
\end{equation}

In fact, the Maslov index can be expressed by the triple index.
\begin{lem} For $\Lambda, \Lambda_0, V(t),t\in[a,b]$ are Lagrangian subspaces. If $V(t)$ transversal to $\Lambda$  for $t\in[a,b]$, then  \bea  \mu(\Lambda_0, V(t);[a,b])=i(V(b),\Lambda_0, \Lambda)-i(V(b),\Lambda_0, \Lambda). \eea
\end{lem}

\begin{proof}

By easy computation
  \bea  \mu(\Lambda_0, V(t);[a,b]) &=&  \mu(\Lambda_0, V(t);[a,b])-  \mu(\Lambda, V(t);[a,b]) \nonumber \\  &=& s(V(a),V(b); \Lambda, \Lambda_0) \nonumber \\  &=& i(V(a),\Lambda, \Lambda_0)-i(V(b),\Lambda, \Lambda_0).   \eea

\end{proof}

 In order to compute $i(Gr(\gamma(T)), \Lambda_0, \Lambda_D)$, we firstly change the basis of  $(\mathbf{C}^{2n}\oplus\mathbf{C}^{2n},-\omega\oplus\omega)$ such that the symplectic structure  with the standard form $J_{2n}$.
  In fact, %$
 let $S=\begin{bmatrix} -I_{n}&0&0&0\\0&0& I_{n}&0\\0&I_{n}&0&0\\0&0&0&I_{n} \end{bmatrix}$, % and  change the basis  of $\mathbf{C}^{2n}\oplus\mathbf{C}^{2n}$ such that the symplectic structure  with the standard form
 then $$S^{T}\mathcal{J}S=J_{2n} .$$
$\Lambda_0$ can be expressed more explicitly.
 Note that the boundary condition \eqref{eq:original general boundary condition} is equivalent to
 $$(-y(0),y(T),x(0),x(T))^T\in \Lambda_0$$ under the new basis.
 % Please keep in mind that in the rest of our paper, we will regard $\Lambda_0$ and $S\Lambda_0$ as the same and still denote by $\Lambda_0$.
 %Let $V(\Lambda_0)$ be the subspace of $\Lambda_N$ defined by $$V(\Lambda_0)=(\Lambda_0+\Lambda_D)\cap\Lambda_N,$$  then $(x(0),x(T))^T\in V$.
%{\color{red}We always denote $$\nu(\Lambda_0)=\dim V(\Lambda_0)$$
 %and assume $\nu(\Lambda_0)=2n-k$.}
  According to the splitting $\mathbf{C}^{2n}\cong \Lambda_N=V\oplus V^\perp$, we can split $(-y(0), y(T))^T=(-y_1(0), y_1(T))^T+(-y_2(0), y_2(T))^T$, where $(-y_1(0), y_1(T))^T\in J_{2n} V^\perp$ and $(-y_2(0), y_2(T))^T\in J_{2n} V$, then we have a matrix $A$ from $V$ to $J_{2n}V$ such that $(-y_2(0), y_2(T))^T=A(x(0),x(T))^T$. Since $\Lambda_0$ is Lagrangian, it's easy to check that $A$ is Hermitian. Now we can choose a suitable bases of $\mathbf{C}^{2n}\oplus\mathbf{C}^{2n}$ such that
  \begin{equation}\label{eq:precise express of boundary condition}
  \Lambda_0=\begin{bmatrix}I_k&0\\0&A\\0&0\\0&I_{2n-k}\end{bmatrix}
  \end{equation}
   under the symplectic form $J_{2n}$. The left column corresponds to $\Lambda_0\cap\Lambda_D$ and the right column corresponds to $\{(Au,u)^T \ | \ u\in V(\Lambda_0)\}$.
Let $\gamma(t)=\begin{bmatrix}D_1(t) &D_2(t)\\D_3(t)&D_4(t) \end{bmatrix}$ be the fundamental solution of  \eqref{eq:Hamilton system},
then $\begin{bmatrix} -I&0\\D_1&D_2\\0&I\\D_3&D_4 \end{bmatrix}$ is a frame of $Gr(\gamma)$ under the symplectic form $J_{2n}$.

Next we will give a simple but useful lemma.
\begin{lem}\label{lem:lemma for computing any symplectic matrix boundary condition}
For  $M\in\Sp(2n)$ and $\Lambda_i\in Lag(\mathbf{C}^{2n}, \omega)$ with $\omega(x,y)=\langle Jx,y\rangle$, then $\Lambda_i\oplus\Lambda_j$ is a Lagrangian subspace of $(\mathbf{C}^{4n},-\omega\oplus\omega)$ for $i,j=1,\cdots, 3$.  We have
\begin{equation}\label{eq:any symplectic matrix boundary condition 1}
i(Gr(M), \Lambda_1\oplus\Lambda_2, \Lambda_1\oplus\Lambda_3)=i(M\Lambda_1, \Lambda_2, \Lambda_3).
\end{equation}
\begin{equation}\label{eq:any symplectic matrix boundary condition 1.1}
i(Gr(M), \Lambda_1\oplus\Lambda_2, \Lambda_3\oplus\Lambda_2)=i(M^{-1}\Lambda_2, \Lambda_1, \Lambda_3 ,-\omega),
\end{equation}
 where the triple index in the  right of \eqref{eq:any symplectic matrix boundary condition 1.1} is defined on $(\mathbf{C}^{2n},-\omega)$.
 \end{lem}
\begin{proof}
We only give details to prove \eqref{eq:any symplectic matrix boundary condition 1}.  In fact, for every $z=(u,Mu)^T\in Gr(M)\cap(\Lambda_1\oplus\Lambda_2+\Lambda_1\oplus\Lambda_3)$, there exist
$u_1, v_1\in\Lambda_1, u_2\in\Lambda_2$ and $v_3\in\Lambda_3$ such that $z=z_1+z_2$, where $z_1=(u_1,u_2)^T, z_2=(v_1,v_3)^T$. It deduces $u=u_1+v_1\in\Lambda_1, Mu=u_2+v_3\in M\Lambda_1\cap(\Lambda_2+\Lambda_3)$. Then by the definition \eqref{eq:the form Q}, we have
$$
\mathfrak{Q}(Gr(M), \Lambda_1\oplus\Lambda_2; \Lambda_1\oplus\Lambda_3)(z,z)=\langle \mathcal{J}z_1,z_2\rangle=\langle-Ju_1,v_1\rangle+\langle Ju_2,v_3\rangle=\langle Ju_2,v_3\rangle.
$$
Therefore, by definition \eqref{trip1} we have $\mathfrak{Q}(Gr(M), \Lambda_1\oplus\Lambda_2; \Lambda_1\oplus\Lambda_3)(z,z)=\mathfrak{Q}(M\Lambda_1, \Lambda_2;\Lambda_3)(u,u)$ with $u \in M\Lambda_1\cap(\Lambda_2+\Lambda_3)$ and consequently we have
 \begin{equation}\label{eq:any symplectic matrix boundary condition 3}
m^+(\mathfrak{Q}(Gr(M),\Lambda_1\oplus\Lambda_2;\Lambda_1\oplus\Lambda_3))=m^+(\mathfrak{Q}(M\Lambda_1 \Lambda_2;\Lambda_3)).
 \end{equation}
 Moreover, one can easily check that
 \begin{equation}\label{eq:any symplectic matrix boundary condition 4}
 \begin{aligned}
 \dim(Gr(M)\cap(\Lambda_1\oplus\Lambda_3))&=\dim(M\Lambda_1\cap\Lambda_3), \\ \dim(Gr(M)\cap(\Lambda_1\oplus\Lambda_2)\cap(\Lambda_1\oplus\Lambda_3))&=\dim(M\Lambda_1\cap\Lambda_2\cap\Lambda_3).
 \end{aligned}
 \end{equation}
  Thus by definition \eqref{trip1} and \eqref{eq:any symplectic matrix boundary condition 3}, \eqref{eq:any symplectic matrix boundary condition 4} we get \eqref{eq:any symplectic matrix boundary condition 1}. The proof of \eqref{eq:any symplectic matrix boundary condition 1.1} is totally similar and we omit the details.
\end{proof}

As a corollary of Lemma \ref{lem:lemma for computing any symplectic matrix boundary condition}, we have:
\begin{cor}\label{cor:lemma for computing any symplectic matrix boundary condition} For  $M\in\Sp(2n)$ and $\Lambda_i\in Lag(\mathbf{C}^{2n}, \omega)$ for $i=1,\cdots,4$. We have
\begin{equation}\label{eq:any symplectic matrix boundary condition 1.2}
i(Gr(M), \Lambda_1\oplus\Lambda_2, \Lambda_3\oplus\Lambda_4)=i(M\Lambda_1, \Lambda_2, \Lambda_4)+i(M^{-1}\Lambda_4,\Lambda_1,\Lambda_3,-\omega),
\end{equation} and equivalently
 \begin{equation}\label{eq:any symplectic matrix boundary condition 1.3}
i(Gr(M), \Lambda_1\oplus\Lambda_2, \Lambda_3\oplus\Lambda_4)=i(M\Lambda_3, \Lambda_2, \Lambda_4)+i(M^{-1}\Lambda_2,\Lambda_1, \Lambda_3,-\omega).
\end{equation}
\end{cor}
\begin{proof} From Lemma \ref{lem:lemma for computing any symplectic matrix boundary condition}  and \eqref{eq:the Hormander index computed by triple index}, we have
\begin{equation}\label{eq:difference for separate 3.1}
\begin{aligned}
& i(Gr(M), \Lambda_1\oplus\Lambda_2, \Lambda_3\oplus\Lambda_4)-i(Gr(M), \Lambda_1\oplus\Lambda_2, \Lambda_1\oplus\Lambda_4)\\&=s(Gr(M), \Lambda_1\oplus\Lambda_2;\Lambda_1\oplus\Lambda_4,\Lambda_3\oplus\Lambda_4)\\
&=i(Gr(M), \Lambda_1\oplus\Lambda_4,\Lambda_3\oplus\Lambda_4)-i(\Lambda_1\oplus\Lambda_2,\Lambda_1\oplus\Lambda_4,\Lambda_3\oplus\Lambda_4)\\
&=i(M^{-1}\Lambda_4, \Lambda_1, \Lambda_3,-\omega).
\end{aligned}
\end{equation}
Since  $i(Gr(M), \Lambda_1\oplus\Lambda_2, \Lambda_1\oplus\Lambda_4)=i(M\Lambda_1,\Lambda_2, \Lambda_4)$, then we have  \eqref{eq:any symplectic matrix boundary condition 1.2}.
Similarly, \begin{equation}\label{eq:difference for separate 3.1}
\begin{aligned}
& i(Gr(M), \Lambda_1\oplus\Lambda_2, \Lambda_3\oplus\Lambda_4)-i(Gr(M), \Lambda_1\oplus\Lambda_2, \Lambda_3\oplus\Lambda_2)\\&=s(Gr(M), \Lambda_1\oplus\Lambda_2;\Lambda_3\oplus\Lambda_2,\Lambda_3\oplus\Lambda_4)\\
&=i(Gr(M), \Lambda_3\oplus\Lambda_2,\Lambda_3\oplus\Lambda_4)-i(\Lambda_1\oplus\Lambda_2,\Lambda_3\oplus\Lambda_2,\Lambda_3\oplus\Lambda_4)\\
&=i(M\Lambda_3, \Lambda_2, \Lambda_4).
\end{aligned}
\end{equation}
Since  $i(Gr(M), \Lambda_1\oplus\Lambda_2, \Lambda_3\oplus\Lambda_2)=i(M^{-1}\Lambda_2,\Lambda_1, \Lambda_3,-\omega)$, then we have  \eqref{eq:any symplectic matrix boundary condition 1.3}. This complete the proof.
\end{proof}

 %Note that the boundary condition \eqref{eq:original general boundary condition} is equivalent to $(-y(0),y(T),x(0),x(T))^T\in \Lambda_0$ under the new %basis.
 % Please keep in mind that in the rest of our paper, we will regard $\Lambda_0$ and $S\Lambda_0$ as the same and still denote by $\Lambda_0$.
 %Let $V(\Lambda_0)$ be the subspace of $\Lambda_N$ defined by $$V(\Lambda_0)=(\Lambda_0+\Lambda_D)\cap\Lambda_N,$$  then $(x(0),x(T))^T\in V$.
%{\color{red}We always denote $$\nu(\Lambda_0)=\dim V(\Lambda_0)$$
 %and assume $\nu(\Lambda_0)=2n-k$.}
 % According to the splitting $\mathbf{C}^{2n}\cong \Lambda_N=V\oplus V^\perp$, we can split $(-y(0), y(T))^T=(-y_1(0), y_1(T))^T+(-y_2(0), y_2(T))^T$, where $(-y_1(0), y_1(T))^T\in J_{2n} V^\perp$ and $(-y_2(0), y_2(T))^T\in J_{4n} V$, then we have a matrix $A$ from $V$ to $J_{2n}V$ such that $(-y_2(0), y_2(T))^T=A(x(0),x(T))^T$. Since $\Lambda_0$ is Lagrangian, it's easy to check that $A$ is Hermitian. Now we can choose a suitable bases of $\mathbf{C}^{2n}\oplus\mathbf{C}^{2n}$ such that $$\Lambda_0=\begin{bmatrix}I_k&0\\0&A\\0&0\\0&I_{2n-k}\end{bmatrix}$$ under the symplectic form $J_{2n}$. The left column corresponds to $(J_{2n}V^\perp,0_{2n})^T=\Lambda_0\cap\Lambda_D$ and the right column corresponds to $\{(Au,u)^T \ | \ u\in % V(\Lambda_0)\}$.
%Let $\gamma(t)=\begin{bmatrix}D_1(t) &D_2(t)\\D_3(t)&D_4(t) \end{bmatrix}$ be the fundamental solution of  \eqref{eq:Hamilton system},
%then $\begin{bmatrix} -I&0\\D_1&D_2\\0&I\\D_3&D_4 \end{bmatrix}$ is a frame of $Gr(\gamma)$ under the symplectic form $J_{2n}$.

\begin{exam}\label{exam:precise computations of example}
Here we will give the precise computations of Example \ref{exm:examples}.

% 1. $\mathbf{Dirichlet \ case}$ By \eqref{eq:particular triple index}, obviously, $i(Gr(I_{2n}), \Lambda_D, \Lambda_D)=0, \Delta(\Lambda_D)=n$.

 %2. $\mathbf{Neumann \ case}$ For every $x=(-u,u,v,v)^T\in Gr(I_{2n})\cap(\Lambda_N+\Lambda_D)$, we have the splitting $x=-y+z$ where $y=(0,0,-v,-v)^T\in \Lambda_N$ and $z=(-u,u,0,0)^T\in \Lambda_D$, then $Q(x,x)=\omega(z,y)=0$ which deduces $m^+(Q( Gr(I_{2n}), \Lambda_N, \Lambda_D)=0$. Moreover, $\dim(Gr(I_{2n})\cap\Lambda_D)=n, \dim(Gr(I_{2n})\cap\Lambda_N\cap\Lambda_D)=0$. So $i(Gr(I_{2n}), \Lambda_N, \Lambda_D)=n, \Delta(\Lambda_D)=0$.

 %3. $\mathbf{Periodic\ case}$ Obviously for every $x\in Gr(I_{2n})\cap(Gr(I_{2n})+\Lambda_D)$ we have $Q(x,x)=0$,
 %then
 %\begin{equation}
% m^+(Q( Gr(I_{2n}), Gr(I_{2n}), \Lambda_D)=0.
  %\end{equation}
  %Together with $\dim(Gr(I_{2n})\cap\Lambda_D)-\dim(Gr(I_{2n})\cap Gr(I_{2n})\cap\Lambda_D)=0$, we conclude that $i(Gr(I_{2n}), \Lambda_P, \Lambda_D)=0, \Delta(\Lambda_P)=n$.

 1. $\mathbf{\cite[Theorem 1.2]{HS}}$   Recall that  $V\subset\Lambda_N$ and  $\Lambda_V=\mathcal{J}V^{\perp}\oplus V$.    For every $x=(-u_1,u_1,v_1,v_1)^T\in Gr(I_{2n})\cap(\Lambda_V+\Lambda_D)$, there exist $y=(u_2,u_3,v_2,v_3)^T\in \Lambda_V$ and $z=(u_4,v_4,0,0)^T\in \Lambda_D$ such that $x=y+z$ which derives $u_2+u_3=-(u_4+v_4), v_1=v_2=v_3$. By direct computations we have
 $
 \mathfrak{Q}(x,x)=\omega(y,z)=0.%\langle -u_4,v_2\rangle-\langle v_4,v_3\rangle=-\langle u_4+v_4, v_2\rangle=\langle u_2+u_3, v_2\rangle=\langle\begin{bmatrix}u_2\\u_3\end{bmatrix},\begin{bmatrix}v_2\\v_3\end{bmatrix}\rangle=0.
$
 So $m^+(\mathfrak{Q}(Gr(I_{2n}), \Lambda_V, \Lambda_D)=0$. Moreover, it's easy to check that $x=(-u,u,v,v)^T\in Gr(I_{4n})\cap\Lambda_V\cap\Lambda_D$ if and only if $(-u,u)^T\in J_{2n}V^\perp$ and $v=0$, then $\dim(Gr(I_{2n})\cap\Lambda_V\cap\Lambda_D)=\dim(V^\perp\cap Gr(-I_n))$, where $Gr(-I_n)=\{(u,-u), u\in\mathbf{C}^n  \}\subset\Lambda_N$.    Now we can conclude that $$i(Gr(I_{2n}), \Lambda_V,\Lambda_D)=n-\dim(V^\perp\cap Gr(-I_n)), \quad  \Delta(\Lambda_V)=\dim(V^\perp\cap Gr(-I_n))).$$

 Particularly, those two concrete cases considered in \cite[Theorem 1.2]{HS} are very easy to be calculated directly.

2. $\mathbf{Separated \ boundary \ condition}$

  By \eqref{eq:particular triple index} and \eqref{eq:any symplectic matrix boundary condition 1.2},  we have
\begin{equation}\label{eq:separated boundary condition}
\begin{aligned}
i(Gr(I_{2n}), \Lambda_s\oplus\Lambda_e, \Lambda_D)&=i(\Lambda^n_D,\Lambda_s, \Lambda^n_D, -\omega)+i(\Lambda_s, \Lambda_e,\Lambda^n_D)\\
&=n-\dim(\Lambda_s\cap\Lambda^n_D)+i(\Lambda_s, \Lambda_e,\Lambda^n_D),
\end{aligned}
\end{equation}
and consequently $\Delta(\Lambda_s\oplus\Lambda_e)=\dim(\Lambda_s\cap\Lambda^n_D)-i(\Lambda_s,\Lambda_e,\Lambda^n_D)$.
 Particularly, for the case where $y(0)=A_sx(0), y(T)=A_ex(T)$, or equivalently, $\Lambda_s=\{(y,x)^T\ |\ y=A_sx\}, \Lambda_e=\{(y,x)^T\ |\ y=A_ex\}$, it's easy to check that $\Lambda_s\cap\Lambda^n_D=\{0\}$. Then by \eqref{trip1} we have

\begin{equation}\label{eq:special separated condition}
i(\Lambda_s,\Lambda_e,\Lambda^n_D)=m^+(\mathfrak{Q}(\Lambda_s,\Lambda_e,\Lambda^n_D)).
\end{equation}
In fact, for every $z=(A_sx,x)^T\in\Lambda_s=\Lambda_e\cap\Lambda^n_D$, there exist $y, u\in \mathbf{C}^n$ such that $\begin{bmatrix}A_sx\\x\end{bmatrix}=\begin{bmatrix}A_ey\\y\end{bmatrix}+\begin{bmatrix}u\\0\end{bmatrix}$. Obviously, we have $x=y, u=A_sx-A_ey=(A_s-A_e)y$. Therefore,
$$\mathfrak{Q}(z,z)=\langle \begin{bmatrix}0&-I_n\\I_n&0\end{bmatrix}\begin{bmatrix}A_ey\\y\end{bmatrix}, \begin{bmatrix}u\\0\end{bmatrix}\rangle=-\langle y,u\rangle=\langle (A_e-A_s)y, y\rangle.$$
Consequently, there holds $m^+(\mathfrak{Q}(\Lambda_s,\Lambda_e,\Lambda^n_D))=m^+(A_e-A_s)$. Now by \eqref{eq:separated boundary condition} and \eqref{eq:special separated condition}, we have
$i(Gr(I_{2n}), \Lambda_s\oplus\Lambda_e, \Lambda_D)=n+m^+(A_e-A_s)$ and consequently $\Delta(Gr(I_{2n}), \Lambda_s\oplus\Lambda_e, \Lambda_D)=-m^+(A_e-A_s)$.
\end{exam}

Next we will give a very simple concrete example to show how our result works for separated boundary condition.
\begin{exam}
Let $\Lambda_s=\{(0,a)^T\ | \ a\in\mathbf{R}\}$ and $\Lambda_e=\{(b,b) \ | \ b\in \mathbf{R}\}$. Consider the Sturm$-$Liouville system
\begin{equation}\label{eq:simple Sturm-Liouville system}
-\ddot{x}(t)-x(t)=0,
\end{equation}
where $x\in W^{2,2}([0,\pi], \mathbf{R})$ and satisfies that $(\dot{x}(0), x(0))^T\in\Lambda_s, (\dot{x}(\pi), x(\pi))^T\in\Lambda_e$.

The first step is to compute the Morse index. Consider the family of Sturm$-$Liouville system with parameter $s\in[0,+\infty)$ as
\begin{equation}\label{eq:family of Sturm-Liouville system}
-\ddot{x}(t)+(s-1)x(t)=0
\end{equation}
under the same boundary condition, by some direct computations we can conclude that there only exist $s_1\in(0,1)$ and $s_2\in(1,+\infty)$ such that system \eqref{eq:family of Sturm-Liouville system} has solutions satisfying the boundary condition. Therefore the Morse index $m^-(x)=2$.

The second step is to compute the Maslov index. Let $y(t)=\dot{x}(t)$ and $z(t)=(y(t),x(t))^T$, then system \eqref{eq:simple Sturm-Liouville system}
can be converted into Hamiltonian system \eqref{eq:Hamilton system} with
$\mathcal{B}(t)=\begin{bmatrix}1&0\\0&1\end{bmatrix}$. The fundamental solution is $\gamma(t)=\begin{bmatrix}\cos t&-\sin t\\ \sin t&\cos t\end{bmatrix}$. Then $\gamma(t)\Lambda_s\cap\Lambda_e\neq\{0\}$ if and only if $-\sin t=\cos t$. Note that there exists only one $t_0\in(0,\pi)$ such that $-\sin t_0=\cos t_0$ and the crossing form $\mathcal{B}(t_0)$ is positive definite, then by formula \eqref{eq:maslov index formula} the Maslov index $\mu^-(z)=1$. Therefore, there holds $\mu^-(z)-m^-(x)=-1$ which coincides with the formula \eqref{eq:formula for special separated boundary condition}.

\end{exam}

\section{Morse index Theorem}

In order to prove the main theorem, we firstly introduce a theorem for the difference of Morse index of Hermitian form with its restriction on a subspace in \S 3.1. We will prove our main results in \S 3.2.

\subsection{An abstract  difference  Morse indices Theorem}
The following work is to introduce the definition of relative Morse index which will play an essential role in the proof of our main theorem. Let $X$ be a complex vector space and $\mathfrak{Q}$ be a Hermitian form on $X$.   For the general case please refer to \cite[Section 3.3]{WL}.  Let $V$ be a subspace of $X$ such that $\dim X/V<+\infty$. Let $V^\mathfrak{Q}$ be the space $\{x\in X|\mathfrak{Q}(x,y) = 0, \forall y \in V\}$. Denote $V^{\mathfrak{Q}\mathfrak{Q}}=(V^{\mathfrak{Q}})^\mathfrak{Q}$. Generally, $\mathfrak{Q}$ induces an Hermitian form  $\overline{\mathfrak{Q}}$ on $X/\ker \mathfrak{Q}$. Let $\mathcal{P}$ be
the natural projection from $X$ to $X/\ker \mathfrak{Q}$. We have $\mathcal{P}(V^\mathfrak{Q}) = (\mathcal{P}(V))^\mathfrak{Q}$ for any subspace $V$. For a given subspace $V$ of $X$ such that $\dim(X/\ker \mathfrak{Q})/(\mathcal{P}V))<+\infty$ and
$V^{\mathfrak{Q}\mathfrak{Q}}=V+\ker \mathfrak{Q}$. We define the relative Morse index $I(\mathfrak{Q}|_{V},\mathfrak{Q})$ as
\begin{equation}\label{eq:definition of relative morse index}
I(\mathfrak{Q}|_{V},\mathfrak{Q})= \dim((V\cap V^\mathfrak{Q} +\ker \mathfrak{Q})/\ker \mathfrak{Q})+m^-(\mathfrak{Q}|_{V^\mathfrak{Q}}).
\end{equation}

	\begin{thm}\label{thm:abstract_re_morse}
		If $V^{\mathfrak{Q}\mathfrak{Q}}=V+\ker(\mathfrak{Q})$ and both $m^-(\mathfrak{Q}), m^-(\mathfrak{Q}|_{V^\mathfrak{Q})}$ exist, then \begin{equation}\label{eq:formula of relative morse index}
		m^-(\mathfrak{Q})-m^-(\mathfrak{Q}|_{V})=I(\mathfrak{Q}|_{V},\mathfrak{Q}).
		\end{equation}
	\end{thm}

	\begin{rem}  In \cite[Equation 1.2]{BTZ}, there is an index theorem give by
\begin{equation}\label{eq:finite dimension index theorem}
\ind (H) = \ind (H|_W) + \ind (H|_{W^\perp}) +\dim(W\cap W^\perp) -\dim(W\cap\ker H),
\end{equation}
where $H$ is a symmetric form on a finite dimensional real vector space $V$ (or a Hermitian form on a complex vector space) and $W\subset V$. Note that $\dim((V\cap V^\mathfrak{Q} +\ker \mathfrak{Q})/\ker \mathfrak{Q})=\dim(V\cap V^\mathfrak{Q})-\dim(V\cap \ker \mathfrak{Q})$, therefore, Theorem \ref{thm:abstract_re_morse} can be viewed as a generalization of index form \eqref{eq:finite dimension index theorem} to infinite dimensional situation.
	\end{rem}	
	In order to prove Theorem \ref{thm:abstract_re_morse}, we need several lemmas.

\begin{lem}\label{lem:Morse_additive}
	Assume that $X=U+W$ and $\mathfrak{Q}(u,w)=0$ with $u\in U$ ,$w\in W$.
	We have $m^-(\mathfrak{Q})=m^-(\mathfrak{Q}|_U)+m^-(\mathfrak{Q}|_W)$.
\end{lem}
\begin{proof}
	Let $U'$, $W'$ be  maximum negative subspaces of $U$, $W$ respectively.
	Then for each $u\in U'$, $w\in W'$ we have $\mathfrak{Q}(u+w,u+w)=\mathfrak{Q}(u,u)+\mathfrak{Q}(w,w)< 0$ if $u+w\neq 0$.
	Let $x\in U'\cap W'$. We have $\mathfrak{Q}(x,x)=0$. It follows that $U'\cap W'={0}$. So $m^-(\mathfrak{Q})\ge \dim(U'+W')=\dim(U')+\dim(W')=m^-(\mathfrak{Q}|_U)+m^-(\mathfrak{Q}|_W)$.
	
	Let $Y\supset U'+W'$ such that $\mathfrak{Q}|_Y <0$. Let $y\in Y\cap (U'+W')^{\mathfrak{Q}}$.
	Since $W\subset U^{\mathfrak{Q}}\subset U'^{\mathfrak{Q}}$, then  $U'^{\mathfrak{Q}}=U'^{\mathfrak{Q}}\cap (U+W)=U'^{\mathfrak{Q}}\cap U+W$. Similarly, we have $W'^{\mathfrak{Q}}=W'^{\mathfrak{Q}}\cap W+U$.
	It follows that  $ (U'+W')^{\mathfrak{Q}}=U'^{\mathfrak{Q}}\cap W'^{\mathfrak{Q}}=(U'^{\mathfrak{Q}}\cap U+W)\cap (W'^{\mathfrak{Q}}\cap W+U)=(W'^{\mathfrak{Q}}\cap W)+(U'^{\mathfrak{Q}}\cap U+W)\cap U=(W'^{\mathfrak{Q}}\cap W)+(U'^{\mathfrak{Q}}\cap U)+(U\cap W)$.
	Then  $y=y_1+y_2+y_3$ with $y_1\in U'^{\mathfrak{Q}}\cap U$, $y_2\in W'^{\mathfrak{Q}}\cap W$, $y_3\in W\cap U$.
	
	Since $U',W'$ are maximum negative subspaces of $U,W$ respectively, we have $\mathfrak{Q}(y_1,y_1)\ge 0$,$\mathfrak{Q}(y_2,y_2)\ge 0$. Note that $\mathfrak{Q}(y_3,y_3)=0$. We have $\mathfrak{Q}(y,y)=\mathfrak{Q}(y_1,y_1)+\mathfrak{Q}(y_2,y_2)+\mathfrak{Q}(y_3,y_3)\ge 0$. Note that $\mathfrak{Q}(y,y)\le 0$ since $\mathfrak{Q}|_Y<0$. Then we have $y=0$.
	It follows that $Y=U'+W'$. Then we have
	$$m^-(\mathfrak{Q})=\dim(U'+W')=m^-(\mathfrak{Q}|_U)+m^-(\mathfrak{Q}|_W).$$
\end{proof}

	\begin{lem}\label{lem:dim_Qorth}
		If $\dim X/V<+\infty$ , then
		$$\dim V^\mathfrak{Q}/\ker \mathfrak{Q}=\dim(X/V^{\mathfrak{Q}\mathfrak{Q}}).$$
		If $\dim W<+\infty$, then
		$$\dim X/W^\mathfrak{Q} =\dim (W+\ker \mathfrak{Q})/\ker \mathfrak{Q} .$$
	\end{lem}
	\begin{proof}
Consider the  sesquilinear forms $\tilde{\mathfrak{Q}}: X/V\times V^\mathfrak{Q} \to \mathbf{C}$ such that $\tilde{\mathfrak{Q}}(x+V,u)=\mathfrak{Q}(x,u)$ with $x+V\in X/V, u\in V^\mathfrak{Q}$. It is well-defined and it induces a finite dimensional linear map $A:X/V\to V^\mathfrak{Q}$ such that $\tilde{\mathfrak{Q}}(p,q)=(Ap,q)$ with $p\in X/V,q\in V^\mathfrak{Q}$. Note that $\ker A=V^{\mathfrak{Q}\mathfrak{Q}}/V$ and $\ker A^*=\ker \mathfrak{Q}$. Then we have $\dim X/V-\dim V^{\mathfrak{Q}\mathfrak{Q}}/V=\dim V^\mathfrak{Q}-\dim\ker \mathfrak{Q}$. It follows that
$$\dim X/(V+\ker \mathfrak{Q})=\dim X/V^{\mathfrak{Q}\mathfrak{Q}}=\dim V^\mathfrak{Q}/\ker \mathfrak{Q}.$$
		
Similarly if we consider sesquilinear form on $X/W^\mathfrak{Q}\times W $, then we get
		$$\dim X/W^\mathfrak{Q} =\dim W/(W\cap \ker \mathfrak{Q})=\dim (W+\ker \mathfrak{Q})/\ker \mathfrak{Q}.$$
\end{proof}
\begin{lem}\label{lem:abstract_re_morse_1}
Assume that $V^{\mathfrak{Q}\mathfrak{Q}}=V$,  $\dim X/V<+\infty$,  $V^\mathfrak{Q}\subset V$. Then we have
		$$m^-(\mathfrak{Q})-m^-(\mathfrak{Q}|_V)=\dim (V^\mathfrak{Q}/\ker \mathfrak{Q}).$$
	\end{lem}
	\begin{proof}
		There is a finite dimensional linear subspace $W\subset X$ such that
		$X=V\oplus W$.
		Then we have $(W^\mathfrak{Q}\cap V^\mathfrak{Q})=(W+V)^\mathfrak{Q}=\ker \mathfrak{Q} $.
		Note that $\ker \mathfrak{Q}\subset V^\mathfrak{Q}\subset V$. We have $W\cap \ker \mathfrak{Q}=\{0\}$.
		Then by Lemma \ref{lem:dim_Qorth}, we have
		$\dim (X/W^\mathfrak{Q})=\dim W$ and $\dim X/V=\dim V^\mathfrak{Q}/\ker \mathfrak{Q} $.
		It follows that
\begin{equation}
\begin{aligned}
\dim X/(V^\mathfrak{Q}+W^\mathfrak{Q})&=\dim X/W^\mathfrak{Q}-\dim (V^\mathfrak{Q}+W^\mathfrak{Q})/W^\mathfrak{Q}=\dim W-\dim V^\mathfrak{Q}/(W^\mathfrak{Q}\cap V^\mathfrak{Q})\\
			&=\dim W-\dim V^\mathfrak{Q}/\ker \mathfrak{Q}=\dim X/V-\dim X/V =0.
\end{aligned}
\end{equation}
		It follows that $X=V^\mathfrak{Q}+W^\mathfrak{Q}$. Then we have $V=(V^\mathfrak{Q}+W^\mathfrak{Q})\cap V=V^\mathfrak{Q}+(V\cap W^\mathfrak{Q})$ and $X=V+W=V^\mathfrak{Q}+W+(V\cap W^\mathfrak{Q})$.
		Note that $(V\cap W^\mathfrak{Q})^\mathfrak{Q}\supset (V^\mathfrak{Q}+W^\mathfrak{Q})$.  	By Lemma \ref{lem:Morse_additive}, we have
	$m^-(\mathfrak{\mathfrak{Q}}|_V)=m^-(\mathfrak{\mathfrak{Q}}|_{W^\mathfrak{\mathfrak{Q}}\cap V})+m^-(\mathfrak{\mathfrak{Q}}|_{V^\mathfrak{\mathfrak{Q}}}) $, and $m^-(\mathfrak{\mathfrak{Q}})=m^-(\mathfrak{\mathfrak{Q}}|_{W^\mathfrak{\mathfrak{Q}}\cap V})+m^-(\mathfrak{\mathfrak{Q}}|_{V^\mathfrak{\mathfrak{Q}}+W})$.
	Since $V^\mathfrak{Q}\subset V$, then $\mathfrak{Q}|_{V^\mathfrak{Q}}=0$, and it follows that $m^-(\mathfrak{Q}|_V)=m^-(\mathfrak{Q}|_{W^\mathfrak{Q}\cap V})$.
		
		Let $\tilde{\mathfrak{Q}}=\mathfrak{Q}|_{V^\mathfrak{Q}+W}$. Then $\ker \tilde{\mathfrak{Q}}=\ker \mathfrak{Q}\subset V^\mathfrak{Q}$.
		Since $V^\mathfrak{Q}\subset V$, we have $\mathfrak{Q}|_{V^\mathfrak{Q}}=0$. It follows that
		$m^\pm(\tilde{\mathfrak{Q}})\ge \dim (V^\mathfrak{Q}/\ker \mathfrak{Q})$. Let $k=\dim W$. Since $\dim V^\mathfrak{Q}/\ker \mathfrak{Q}=\dim W$, we have $2k\le m^-(\tilde{\mathfrak{Q}})+m^+(\tilde{\mathfrak{Q}})\le \dim (V^\mathfrak{Q}+W)/\ker \mathfrak{Q}=2k$. It follows that $m^-(\tilde{\mathfrak{Q}})=k=\dim (V^\mathfrak{Q}/\ker \mathfrak{Q})$.
	\end{proof}
	Then we can prove Theorem \ref{thm:abstract_re_morse}:
	\begin{proof}[Proof of Theorem \ref{thm:abstract_re_morse}]
		Let $W=V+V^\mathfrak{Q}$. Then there is a linear subspace $U\subset V^\mathfrak{Q}$ such that $W=V\oplus U$. Since $\dim (X/V) <+\infty$, we see that $\dim U<\infty $.
		For each $x\in V, y\in U\subset V^\mathfrak{Q}$, we have $\mathfrak{Q}(x,y)=0$.
		It follows that $m^-(\mathfrak{Q}|_W)-m^-(\mathfrak{Q}|_{V})=m^-(U)$.
		Note that $V^\mathfrak{Q}=V^\mathfrak{Q}\cap(U+V)=(V\cap V^\mathfrak{Q})\oplus U$.
		Since $V^\mathfrak{Q}\cap V \subset V^{\mathfrak{Q}\mathfrak{Q}} $, we have $m^-(\mathfrak{Q}|_{V^\mathfrak{Q}})=m^-(\mathfrak{Q}|_U)$.
		Then we have $m^-(\mathfrak{Q}|_{V+V^\mathfrak{Q}})-m^-(\mathfrak{Q}|_{V})=m^-(\mathfrak{Q}|_{V^\mathfrak{Q}})$.
		
		Since $V^{\mathfrak{Q}\mathfrak{Q}}=V+\ker \mathfrak{Q}$, we have $V+V^\mathfrak{Q}=V^{\mathfrak{Q}\mathfrak{Q}}+V^\mathfrak{Q}$. By Lemma \ref{lem:dim_Qorth}, we have
\begin{equation}
\begin{aligned}
\dim X/(V^{\mathfrak{Q}\mathfrak{Q}}+V^\mathfrak{Q})&=\dim(X/V^{\mathfrak{Q}\mathfrak{Q}})-\dim(V^{\mathfrak{Q}\mathfrak{Q}}+V^\mathfrak{Q})/V^{\mathfrak{Q}\mathfrak{Q}}=\dim V^\mathfrak{Q}/\ker \mathfrak{Q}-\dim V^\mathfrak{Q}/(V^\mathfrak{Q}\cap V^{\mathfrak{Q}\mathfrak{Q}})\\
&=\dim (V^\mathfrak{Q}\cap V^{\mathfrak{Q}\mathfrak{Q}})/\ker \mathfrak{Q}=\dim (X/(V^{\mathfrak{Q}\mathfrak{Q}}+V^\mathfrak{Q})^{\mathfrak{Q}\mathfrak{Q}}).
\end{aligned}
\end{equation}
It follows that $(V+V^\mathfrak{Q})^{\mathfrak{Q}\mathfrak{Q}}=V+V^\mathfrak{Q}$. Then by Lemma \ref{lem:abstract_re_morse_1}, we have
		$$ m^-(\mathfrak{Q})-m^-(\mathfrak{Q}|_{V+V^\mathfrak{Q}})=\dim (V^{\mathfrak{Q}\mathfrak{Q}}\cap V^\mathfrak{Q})/\ker \mathfrak{Q}=\dim (V\cap V^\mathfrak{Q}+\ker \mathfrak{Q})/\ker \mathfrak{Q}.$$
\end{proof}

	\begin{lem}
		We assume that $X$ is a Hilbert space, $H$ is a closed subspace of $X$,
		and $\mathfrak{Q}(x,y)=(Ax,y)$ with bounded Fredholm self-adjoint operator $A$.
		Then we have
		\begin{equation}\label{lem perp}
		H^{\mathfrak{Q}\mathfrak{Q}}=H+\ker \mathfrak{Q}
		\end{equation}
	\end{lem}
	\begin{proof}
		Since $A$ is a Fredholm self-adjoint operator, then
		$\dim\ker A<\infty$, $\ker A=(\image A)^\perp$ and $\ker A=\ker \mathfrak{Q}$.
		Since $(Ax,y)=(x,Ay)$, we have
		$H^\mathfrak{Q}=(AH)^\perp=A^{-1}(H^\perp)=A^{-1}(H^\perp\cap \image A)$.
		It follows that
		$$H^{\mathfrak{Q}\mathfrak{Q}}=(A(A^{-1}(H^\perp\cap \image A)))^\perp =(H^\perp\cap \image A)^{\perp}=\overline{H^{\perp\perp}+\image A^{\perp}}=\overline{H^{\perp\perp}+\ker A}.$$
		Since $H$ is closed and $\dim \ker A<+\infty$, then $H^{\perp\perp}=H$ and $(H+\ker A)$ is closed. Then we have
		$$H^{\mathfrak{Q}\mathfrak{Q}}=\overline{H^{\perp\perp}+\ker A}=H+\ker \mathfrak{Q} .$$
	\end{proof}

\subsection{Morse index Theorem}

Let $A$ be the Hermitian matrix defined by $\Lambda_0$.
For a stationary point $x$ of system \eqref{eq:L-S}, its index form $I_x$ is given by
\begin{equation}\label{eq:index form}
I_x(\xi,\eta)=\int_{ 0 }^{ T }\{\langle P\dot{\xi},\dot{\eta}\rangle+\langle Q\xi,\dot{\eta}\rangle+\langle Q^{T}\dot{\xi},\eta\rangle+\langle R\xi,\eta\rangle\} dt-\langle A\begin{bmatrix} \xi(0)\\ \xi(T) \end{bmatrix},\begin{bmatrix} \eta(0)\\ \eta(T) \end{bmatrix}\rangle
 \end{equation}
 on $$H_{\Lambda_0}([0,T])=\{\xi\in W^{1,2}([0,T],\mathbf{C}^n)\ | \ (\xi(0),\xi(T))\in V(\Lambda_0)\}.$$ For convenience, we will drop the subscript $x$ of $I_x$.  By \cite[Theorem A2]{M}, the space $W^{1,2}(0,T)$ can be  continuously imbedded into
 the space
 $$
 \mathbf{C}^{0}=\{u\in \mathbf{C}^{0}(0,T) \ | \  \|u\| _{\mathbf{C}^0}=\sup_{0<t<T}|u(t)|<\infty\}
$$
  and for any $\varepsilon >0$, there exists a constant $C_{\varepsilon}$ such that
  \begin{equation}\label{eq:immbed inequality}
\|u\| _{\mathbf{C}^0}\leq \varepsilon \|\dot{u}\|_{L^{2}(0,T)}
  +C_{\varepsilon}\|u\|_{L^{2}(0,T)}.
  \end{equation}
By \eqref{eq:index form} and \eqref{eq:immbed inequality} we have
\begin{equation}\label{eq:equivalent norm 1}
\begin{aligned}
I(\xi,\xi)&\geq c_1\Vert\dot{\xi}\Vert^2_{L^2}+c_2\Vert\xi\Vert_{L^2}\Vert\dot{\xi}\Vert_{L^2}+c_3\Vert\xi\Vert^2_{L^2}
-\varepsilon_1\Vert\dot{\xi}\Vert_{L^2}^2-C_{\varepsilon_1}\Vert\xi\Vert_{L^2}^2\\
&\geq c_1\Vert\dot{\xi}\Vert^2_{L^2}-\varepsilon_2c_2\Vert\dot{\xi}\Vert_{L^2}^2-C_{\varepsilon_2}c_2\Vert\xi\Vert_{L^2}^2+c_3\Vert\xi\Vert^2_{L^2}
-\varepsilon_1\Vert\dot{\xi}\Vert_{L^2}^2-C_{\varepsilon_1}\Vert\xi\Vert_{L^2}^2\\
&=c_4(\Vert\dot{\xi}\Vert^2_{L^2}+\Vert\xi\Vert_{L^2}^2)+c_5\Vert\xi\Vert_{L^2}^2,
\end{aligned}
\end{equation}
 where $c_i, i=1,\ldots, 5$ and $\varepsilon_i, C_{\varepsilon_i}, i=1,2$ are constants such that $c_1, c_4>0, \varepsilon_i$ is small enough. Moreover,  the second inequality comes from Young inequality. Similarly, for some suitable constants $C_4, C_5$ we have
 \begin{equation}\label{eq:equivalent norm 2}
I(\xi,\xi)\leq C_4(\Vert\dot{\xi}\Vert^2_{L^2}+\Vert\xi\Vert_{L^2}^2)+C_5\Vert\xi\Vert_{L^2}^2.
\end{equation}
 Then we can choose a constant $C_{H_{\Lambda_0}}$ large enough such that $I(\xi,\xi)+C_{H_{\Lambda_0}}\Vert\xi\Vert_{L^2}^2$ is a norm equivalent with the $W^{1,2}$-norm on $H_{\Lambda_0}([0,T])$ and will be denoted by $\Vert\cdot\Vert_{H_{\Lambda_0}}$.

 Recall the definition of operator $\mathcal{A}$ in \eqref{eq:S-L equation} and integral by parts, then for every $\xi\in E_{\Lambda_0}(0,T)$ and $\eta\in H_{\Lambda_0}([0,T])$ we have
 \begin{equation}\label{eq:relation between operator and index form}
\begin{aligned}
 \langle \mathcal{A}\xi,\eta\rangle_{L^2}&=\int_{ 0 }^{ T }\langle-\frac{d}{dt}( P\dot{\xi}+Q\xi)+Q^{T}\dot{\xi}+R\xi,\eta\rangle dt\\
 &=\int_{ 0 }^{ T }\{\langle P\dot{\xi},\dot{\eta}\rangle+\langle Q\xi,\dot{\eta}\rangle+\langle Q^{T}\dot{\xi},\eta\rangle +\langle R\xi,\eta\rangle\} dt -\langle\begin{bmatrix} -y(0)\\y(T) \end{bmatrix},\begin{bmatrix} \eta(0)\\ \eta(T) \end{bmatrix}\rangle\\
 &=\int_{ 0 }^{ T }\{\langle P\dot{\xi},\dot{\eta}\rangle+\langle Q\xi,\dot{\eta}\rangle+\langle Q^{T}\dot{\xi},\eta\rangle+\langle R\xi,\eta\rangle\} dt-\langle A\begin{bmatrix} \xi(0)\\ \xi(T) \end{bmatrix},\begin{bmatrix} \eta(0)\\ \eta(T) \end{bmatrix}\rangle\\
&=I(\xi,\eta),
\end{aligned}
\end{equation}
where $y(t)=P(t)\dot{\xi}(t)+Q(t)\xi(t)$. The third equation holds since $\xi$ satisfies the boundary condition $\Lambda_0$ in the form \eqref{eq:precise express of boundary condition}, therefore we can decompose $\begin{bmatrix} -y(0)\\y(T) \end{bmatrix}=\begin{bmatrix} -y_{1}(0)\\y_{1}(T) \end{bmatrix}+
\begin{bmatrix} -y_{2}(0)\\y_{2}(T) \end{bmatrix}$ such that $(-y_1(0), y_1(T))^T\in J_{2n} V^\perp$ and $(-y_2(0), y_2(T))^T=A(\xi(0),\xi(T))^T \in J_{2n} V$.

Let $m^-(I,\Lambda_0)$ be the dimension of the maximal negative definite subspace of the index form $I$ on $H_{\Lambda_0}([0,T])$.  The following theorem is standard, but for reader's convenience, we give details of the proof.
\begin{thm}\label{defi:Morse index}
For a critical point $x$ of the Lagrangian system \eqref{eq:L-S} with the boundary condition \eqref{eq:original general boundary condition}, \bea \dim\ker(I)=m^0(\mathcal{A}), \quad  m^-(I,\Lambda_0)=m^-(\mathcal{A}). \eea

\end{thm}
\begin{proof}

For the first equation we only need to prove that $\xi\in \ker I$ is equivalent to $\xi\in \ker \mathcal{A}$.

$\Leftarrow$ \quad It's obvious by \eqref{eq:relation between operator and index form}.

 $\Rightarrow$ \quad Assuming that $\xi \in \ker (I)$, then $I(\xi,\eta)=0$ for any $\eta \in H_{\Lambda_0}([0,T])$. So for every $\eta\in H_{\Lambda_0}([0,T])$ such that $\eta(0)=\eta(T)=0$, by \eqref{eq:index form} we have
 \begin{equation}
 \begin{aligned}
0=I(\xi,\eta)&= \int_{ 0 }^{ T }\{\langle P\dot{\xi},\dot{\eta}\rangle+\langle Q\xi,\dot{\eta}\rangle+\langle Q^{T}\dot{\xi},\eta\rangle+\langle R\xi,\eta\rangle\} dt -\langle A\begin{bmatrix} \xi(0)\\ \xi(T) \end{bmatrix},\begin{bmatrix} \eta(0)\\ \eta(T) \end{bmatrix}\rangle\\
&=\int_{ 0 }^{ T }\langle P\dot{\xi}+Q\xi,\dot{\eta}\rangle dt+\int_{ 0 }^{ T }\langle -\int_{ 0 }^{ t }(Q^{T}\dot{\xi}+R\xi)ds,\dot{\eta}\rangle dt\\
&=\int_{ 0 }^{ T }\langle P\dot{\xi}+Q\xi-\int_{ 0 }^{ t }(Q^{T}\dot{\xi}+R\xi)ds,\dot{\eta}\rangle dt.
\end{aligned}
\end{equation}
Then by $du \ Bois-Reymond $ theorem we have
 \begin{equation}\label{eq:variation of index form}
 P\dot{\xi}+Q\xi-\int_{ 0 }^{ t }(Q^{T}\dot{\xi}+R\xi)ds=constant
\end{equation}
Derivative on both sides of \eqref{eq:variation of index form} leads that $\xi$ is a solution of \eqref{eq:S-L equation}.
The rest is to prove $x$ satisfies condition \eqref{eq:original general boundary condition}.
Decomposing $\begin{bmatrix} -y(0)\\y(T) \end{bmatrix}$ into $\begin{bmatrix}-y(0)\\y(T) \end{bmatrix}=
\begin{bmatrix}-y_{1}(0)\\y_{1}(T) \end{bmatrix}+\begin{bmatrix}-y_{2}(0)\\y_{2}(T) \end{bmatrix}$,
where $\begin{bmatrix}-y_{1}(0)\\y_{1}(T) \end{bmatrix}\in J_{2n}V^{\perp}$ and $\begin{bmatrix}-y_{2}(0)\\y_{2}(T) \end{bmatrix}\in J_{2n}V$,
then by \eqref{eq:index form} we have
%\begin{equation}
$$
I(x,\eta)=
\langle \begin{bmatrix}-y_{1}(0)\\y_{1}(T) \end{bmatrix}+\begin{bmatrix}-y_{2}(0)\\y_{2}(T) \end{bmatrix}-A\begin{bmatrix} \xi(0)\\ \xi(T) \end{bmatrix},\begin{bmatrix} \eta(0)\\ \eta(T) \end{bmatrix}\rangle.
$$
%\end{equation}
Since $\eta\in H_{\Lambda_0}([0,T])$ is arbitrary, then we can take $\eta$ such that
$\begin{bmatrix}-y_{2}(0)\\y_{2}(T) \end{bmatrix}-A\begin{bmatrix} \xi(0)\\ \xi(T) \end{bmatrix}=\begin{bmatrix} \eta(0)\\ \eta(T) \end{bmatrix}$,
then $\begin{bmatrix}-y_{2}(0)\\y_{2}(T) \end{bmatrix}-A\begin{bmatrix} \xi(0)\\ \xi(T) \end{bmatrix}=0$. This means the condition \eqref{eq:original general boundary condition} holds.

 In order to prove the second equation we decompose $E_{\Lambda_0}(0,T)$ into $E^{+}\oplus E^{0} \oplus E^{-}$, where $E^{+},E^{0}$ and $E^{-}$ are the eigenspaces of $\mathcal{A}$ in $E_{\Lambda_0}(0,T)$ corresponding to the positive, zero, negative eigenvalues respectively. For every eigenvector $\eta$ corresponding to a negative eigenvalue $\lambda$ of $\mathcal{A}$, by \eqref{eq:relation between operator and index form} we have
$$
 I(\eta,\eta)=\langle\mathcal{A}\eta, \eta\rangle=\langle\lambda\eta, \eta\rangle<0.
$$
Therefore, $m^-(I,\Lambda_0)\geq m^-(\mathcal{A})$.

Let $H_{\Lambda_0}([0,T])^-$ be the maximal negative subspace of index form $I$ and $\{e_1, \ldots, e_k\}$ be the orthogonal bases of $E^-$ and $\{e_{k+1}, e_{k+2},  \ldots\}$ be the orthogonal bases of $E^0\oplus E^+$. If $\xi\in H_{\Lambda_0}([0,T])$ satisfies that
\begin{equation}
\begin{aligned}
0&=\langle e_i, \xi\rangle_{H_{\Lambda_0}}=I(e_i, \xi)+C_{H_{\Lambda_0}}\langle e_i, \xi\rangle_{L^2}\\
&=\langle\mathcal{A}e_i, \xi\rangle_{L^2}+C_{H_{\Lambda_0}}\langle e_i, \xi\rangle_{L^2}=(\lambda_i+C_{H_{\Lambda_0}})\langle e_i, \xi\rangle_{L^2}, \quad \forall \ e_i,
\end{aligned}
\end{equation}where $\langle \cdot, \cdot\rangle_{H_{\Lambda_0}}$ is the inner product induced by $\Vert\cdot\Vert_{H_{\Lambda_0}}$, then there must hold
$\xi=0$. Therefore, we have
$$
span\{e_1, \ldots, e_k\}\oplus\overline{span\{e_{k+1}, e_{k+2},\ldots\}}=H_{\Lambda_0}([0,T])
$$
 Then if  $m^-(I,\Lambda_0)> m^-(\mathcal{A})$, there must exist $\xi\neq0\in H_{\Lambda_0}([0,T])^-\cap (E^{0} \oplus E^{+})$. But by \eqref{eq:index form} we have $I(\xi,\xi)\geq0$ which is a contradiction. Therefore there must hold $m^-(I,\Lambda_0)= m^-(\mathcal{A})$. We conclude the proof.

 \end{proof}

%\section{Proof of main results}

%Before to prove our main results, we need some preliminaries.

%\begin{lem}
%If $\mathcal{V}^{\perp\perp}=V+\ker(Q)$ and both $m^-(Q), m^-(Q|_{\mathcal{V}^\perp)}$ exist, then \begin{equation}\label{eq:formula of relative morse index}
%m^-(Q)-m^-(Q|_{\mathcal{V}})=I(Q|_{\mathcal{V}},Q).
%\end{equation}
%\end{lem}
%\begin{proof}
%\end{proof}

%\begin{rem}  \cite{BTZ} on closed geodesics

%\end{rem}

%\begin{lem}
%Under some condition  \begin{equation}\label{lem perp}
%H^{\perp\perp}=H+\ker Q
%\end{equation}
%\end{lem}
%\begin{proof}
%\end{proof}

Now we will give the direct proof of Theorem \ref{thm:difference of two morse index}:
\begin{proof}[Proof of Theorem \ref{thm:difference of two morse index}]
Let $H_0=\{\xi\in W^{1,2}([0,T],\mathbf{C}^{n})\ | \ \xi(0)=\xi(T)=0\}$ and $H_0^I=\{\xi\in H \ | \ I(\xi,\eta)=0, \forall \eta \in H_0\}$. Obviously, if $\xi\in W^{2,2}([0,T],\mathbf{C}^{n})$, then integral by parts of index form \eqref{eq:index form}, we have
\begin{equation}\label{eq:part integral of index form}
 I(\xi,\eta)=\int_{ 0 }^{ T }\langle-\frac{d}{dt}( P\dot{\xi}+Q\xi)+Q^{T}\dot{\xi}+R\xi,\eta\rangle dt +
\langle \begin{bmatrix} -y(0)\\y(T) \end{bmatrix}-A\begin{bmatrix} \xi(0)\\ \xi(T) \end{bmatrix},\begin{bmatrix} \eta(0)\\ \eta(T) \end{bmatrix}\rangle,
\end{equation}
where $y(t)=P(t)\dot{\xi}(t)+Q(t)\xi(t)$. Actually, $H_0^I$ can be expressed in a specific way by
%\begin{equation}
 $$
 H_0^I=\{\xi\in H \ | \  -\frac{d}{dt}(P(t)\dot{\xi}(t)+Q(t)\xi(t))+Q(t)^{T}\dot{\xi}(t)+R(t)\xi(t)=0\}.
 $$
%\end{equation}
%By \eqref{eq:part integral of index form} again, for every given $\eta\in H_0^{\perp\perp}$ we have
%\begin{equation}\label{eq:part integral of index form to prove othogonal is self}
%\begin{aligned}
% 0=I(\xi,\eta)&=\int_{ 0 }^{ T }\langle-\frac{d}{dt}( P\dot{\xi}+Q\xi)+Q^{T}\dot{\xi}+R\xi,\eta\rangle dt +
%\langle \begin{bmatrix} -y(0)\\y(T) \end{bmatrix}-A\begin{bmatrix} \xi(0)\\ \xi(T) \end{bmatrix},\begin{bmatrix} \eta(0)\\ \eta(T) \end{bmatrix}\rangle\\
%&=\langle \begin{bmatrix} -y(0)\\y(T) \end{bmatrix}-A\begin{bmatrix} \xi(0)\\ \xi(T) \end{bmatrix},\begin{bmatrix} \eta(0)\\ \eta(T) \end{bmatrix}\rangle
%\end{aligned}
%\end{equation}
%for every $\xi\in H_0^{\perp}$. Since we can decompose $\begin{bmatrix} -y(0)\\y(T) \end{bmatrix}$ into $\begin{bmatrix} -y_1(0)\\y_1(T) \end{bmatrix}+\begin{bmatrix} -y_2(0)\\y_2(T) \end{bmatrix}$ such that $\begin{bmatrix} -y_1(0)\\y_1(T) \end{bmatrix}\in V^\perp$ and $\begin{bmatrix} -y_2(0)\\y_2(T) \end{bmatrix}\in V$. Note that one can choose a suitable $\xi$ such that $\begin{bmatrix} -y_2(0)\\y_2(T) \end{bmatrix}-A\begin{bmatrix} \xi(0)\\ \xi(T) \end{bmatrix}=\begin{bmatrix} \eta(0)\\ \eta(T) \end{bmatrix}$, then by \eqref{eq:part integral of index form to prove othogonal is self} $\begin{bmatrix} \eta(0)\\ \eta(T) \end{bmatrix}=0$. Therefore we have $\eta\in H_0^{\perp\perp}=H_0$.

Recall the discussions in \eqref{eq:equivalent norm 1}, \eqref{eq:equivalent norm 2}, we can choose a constant $C_{H_{\Lambda_0}}$ large enough such that $I(\xi,\xi)+C_{H_{\Lambda_0}}\Vert\xi\Vert_{L^2}^2$ is a norm equivalent with the $W^{1,2}$-norm on $H_{\Lambda_0}([0,T])$ and will be denoted by
$\Vert\cdot\Vert_{H_{\Lambda_0}}$. Moreover, we will denote $\langle \cdot, \cdot\rangle_{H_{\Lambda_0}}$ the inner product induced by $\Vert\cdot\Vert_{H_{\Lambda_0}}$. Note that the injection from $W^{1,2}(0,T)$ to $L^2(0,T)$ is compact, then for every $\xi,\eta\in H_{\Lambda_0}$, we have $I(\xi,\eta)=\langle \xi, \eta\rangle_{H_{\Lambda_0}}-C_{H_{\Lambda_0}}\Vert\xi\Vert_{L^2}^2=\langle (\mathcal{I}+\mathcal{T})\xi, \eta\rangle_{H_{\Lambda_0}}$ with compact operator $\mathcal{T}$. Obviously, $\mathcal{I}+\mathcal{T}$ is a bounded self-adjoint Fredholm operator. Since  $H_0$ is a close subspace of $H_{\Lambda_0}([0,T])$, then by Lemma \ref{lem:dim_Qorth} we have $H_0^{II}=H_0+\ker I$.  By Theorem \ref{thm:abstract_re_morse} we have
\begin{equation}\label{eq:formula of computing difference of two morse index}
m^-(\Lambda_0)-m^-(\Lambda_D)=m^-(I|_{H_0^I})+\dim((H_0\cap H_0^I+\ker I )/\ker I).
\end{equation}
For every $\xi\in H_0^I$, let $p=\begin{bmatrix} -y(0)\\y(T) \end{bmatrix}, q=\begin{bmatrix} \xi(0)\\\xi(T) \end{bmatrix}$ and $z=\begin{bmatrix} p\\q \end{bmatrix}$, then by \eqref{eq:part integral of index form} we have
%\begin{equation}
$$
 I(\xi,\xi)=\langle p-Aq,q\rangle
 =\langle\begin{bmatrix} 0&-I_{2n}\\I_{2n}&0\end{bmatrix}\begin{bmatrix} p-Aq\\0 \end{bmatrix},\begin{bmatrix} p\\q \end{bmatrix}\rangle.
$$
%\end{equation}
Note that $\begin{bmatrix} p\\q \end{bmatrix} \in Gr(\gamma(T))\cap (\Lambda_0+\Lambda_D)$. So there exists some $p_1\in J_{2n}V^{\perp}$ such that
$\begin{bmatrix} p_1+Aq\\q\end{bmatrix}\in \Lambda_0$ and the split
$\begin{bmatrix} p\\q \end{bmatrix}=\begin{bmatrix} p_1+Aq\\q\end{bmatrix}+\begin{bmatrix} p-p_1-Aq\\0 \end{bmatrix}$ holds. Since $\langle\begin{bmatrix} 0&-I_{2n}\\I_{2n}&0\end{bmatrix}\begin{bmatrix} -p_1\\0 \end{bmatrix},\begin{bmatrix} p\\q \end{bmatrix}\rangle=0$, then by the definition of triple index there holds
%\begin{equation}
$$
 I(\xi,\xi)=\langle\begin{bmatrix} 0&-I_{2n}\\I_{2n}&0\end{bmatrix}\begin{bmatrix} p-Aq-p_1\\0 \end{bmatrix},\begin{bmatrix} p\\q \end{bmatrix}\rangle=
 -\mathfrak{Q}(z,z).
 $$
%\end{equation}
and therefore we have
\begin{equation}\label{eq:morse and triple}
 m^-(I|_{H_0^I})=m^-(-\mathfrak{Q}(Gr(\gamma(T), \Lambda_0; \Lambda_D))=m^+(\mathfrak{Q}(Gr(\gamma(T), \Lambda_0; \Lambda_D)).
\end{equation}
Note that
\begin{equation}\label{eq:equation of dimension}
\begin{aligned}
\dim((H_0\cap H_0^I+\ker I )/\ker I)&=\dim(H_0\cap H_0^I)-\dim(H_0\cap H_0^I\cap\ker I )\\
&=\dim(Gr(\gamma(T))\cap\Lambda_0)-\dim(Gr(\gamma(T))\cap\Lambda_0\cap\Lambda_D),
\end{aligned}
\end{equation}
then by \eqref{eq:compute the triple index}, \eqref{eq:formula of computing difference of two morse index}, \eqref{eq:morse and triple} and \eqref{eq:equation of dimension}, we have
\begin{equation}
m^-(\Lambda_0)-m^-(\Lambda_D)=i(Gr(\gamma(T)),\Lambda_0,\Lambda_D).
 \end{equation}
 This complete the proof.
\end{proof}

Now we can present the proof of Corollary \ref{cor:relationship between M-M index}.
\begin{proof}[Proof of Corollary \ref{cor:relationship between M-M index}]
Denote the Maslov index for Dirichlet boundary condition by $\mu^-_d(z)$, then as well-known that (for example \cite[Remark 3.6]{HS} ) $\mu^-_d(z)-m^-(\Lambda_D)=n$. By direct computations, we have
\begin{equation}
\begin{aligned}
m^-(x)-\mu^-(z)&=(m^-(x)-m^-(\Lambda_D))+(m^-(\Lambda_D)-\mu^-_d(z))+(\mu^-_d(z)-\mu^-(z))\\
&=i(Gr(\gamma(T)),\Lambda_0,\Lambda_D)-n+s(Gr(I_{2n}),Gr(\gamma(T));\Lambda_0,\Lambda_D)\\
&=i(Gr(\gamma(T)),\Lambda_0,\Lambda_D)-n+i(Gr(I_{2n}),\Lambda_0,\Lambda_D)-i(Gr(\gamma(T)),\Lambda_0,\Lambda_D)\\
&=-n+i(Gr(I_{2n}),\Lambda_0,\Lambda_D).
\end{aligned}
\end{equation}
This means $\mu^-(z)-m^-(x)=n-i(Gr(I_{2n}),\Lambda_0,\Lambda_D)$ which completes the proof.
\end{proof}

 \begin{proof}[Proof of Theorem \ref{thm: morse index thm gene}]
By Theorem \ref{thm:difference of two morse index}, we have
\begin{equation}\label{eq:lambda conjugate theorem 1}
\begin{aligned}
m^-(\Lambda_s\oplus\Lambda_e)-m^-(\Lambda_s\oplus\Lambda^n_D)&=(m^-(\Lambda_s\oplus\Lambda_e)-m^-(\Lambda_D))-
(m^-(\Lambda_s\oplus\Lambda^n_D)-m^-(\Lambda_D))\\
&=i(Gr(\gamma(T)),\Lambda_s\oplus\Lambda_e, \Lambda_D))-i(Gr(\gamma(T)),\Lambda_s\oplus\Lambda^n_D, \Lambda_D))\\
&=i(\gamma(T))\Lambda_s,\Lambda_e, \Lambda^n_D),
\end{aligned}
\end{equation}
where the last equality is from \eqref{eq:any symplectic matrix boundary condition 1.1}, \eqref{eq:any symplectic matrix boundary condition 1.2}.

%By the same discussions for separated boundary conditions in example \ref{exam:precise computations of example} and Lemma \ref{lem:lemma for computing any symplectic matrix boundary condition} we have
%\begin{equation}\label{eq:lambda conjugate theorem 2}
%\begin{aligned}
%&i(Gr(\gamma(T)),\Lambda_s\oplus\Lambda_e, \Lambda_D))-i(Gr(\gamma(T)),\Lambda_s\oplus\Lambda^n_D, \Lambda_D))\\
%&=i(Gr(\gamma(T)),\Lambda_s\oplus\Lambda_e, \Lambda_s\oplus\Lambda^n_D)=i(\gamma(T))\Lambda_s,\Lambda_e, \Lambda^n_D).
%\end{aligned}
%\end{equation}
Then \eqref{eq:lambda conjugate theorem 1} %and \eqref{eq:lambda conjugate theorem 2}
derives that
\begin{equation}\label{eq:lambda conjugate theorem 3}
m^-(\Lambda_s\oplus\Lambda_e)=m^-(\Lambda_s\oplus\Lambda^n_D)+i(\gamma(T))\Lambda_s,\Lambda_e, \Lambda^n_D).
\end{equation}
The rest work is to compute $m^-(\Lambda_s\oplus\Lambda^n_D)$ by using Maslov index. In fact, for every $ t\in[0,T]$ and $z=\gamma(t)^{-1}(y,0)^T\in\gamma(t)\Lambda_s\cap\Lambda^n_D$, by some direct computations the crossing form is given by
%\begin{equation}\label{eq:lambda conjugate theorem 4}
$$
\langle -\gamma(t)^{T}J\dot{\gamma}(t)z,z\rangle=\langle \mathcal{B}(t)\begin{bmatrix}y\\0\end{bmatrix},\begin{bmatrix}y\\0\end{bmatrix}\rangle=\langle P^{-1}(t)y,y\rangle>0,
$$
%\end{equation}
where $\mathcal{B}(t)$ is defined in \eqref{eq:the form of B}. Moreover, one can easily check that $Gr(\gamma(0))\cap(\Lambda_s\oplus\Lambda^n_D)=Gr(I_{2n})\cap(\Lambda_s\oplus\Lambda^n_D)=\dim(\Lambda_s\cap\Lambda^n_D)$ and $Gr(\gamma(t))\cap(\Lambda_s\oplus\Lambda^n_D)=\dim(\gamma(t)\Lambda_s\cap\Lambda^n_D)$ for every $t\in(0,T)$. Then by the formula \eqref{eq:maslov index formula} we have
\begin{equation}\label{eq:lambda conjugate theorem 5}
\mu(\Lambda_s\oplus\Lambda^n_D, Gr(\gamma(t)))=\sum_{0<t<T}\dim(\gamma(t)\Lambda_s\cap\Lambda^n_D)+\dim(\Lambda_s\cap\Lambda^n_D).
\end{equation}
By Corollary \ref{cor:relationship between M-M index}, we have
\begin{equation}\label{eq:lambda conjugate theorem 6}
m^-(\Lambda_s\oplus\Lambda^n_D)=\mu(\Lambda_s\oplus\Lambda^n_D, Gr(\gamma(t)))-n+i(Gr(I_{2n}),\Lambda_s\oplus\Lambda^n_D, \Lambda_D).
\end{equation}
Recall that  $i(Gr(I_{2n}),\Lambda_s\oplus\Lambda^n_D, \Lambda_D)=n-\dim(\Lambda_s\cap\Lambda^n_D)$, together with \eqref{eq:lambda conjugate theorem 5} then \eqref{eq:lambda conjugate theorem 6} is converted into
\begin{equation}\label{eq:lambda conjugate theorem 9}
m^-(\Lambda_s\oplus\Lambda^n_D)=\sum_{0<t<T}\dim(\gamma(t)\Lambda_s\cap\Lambda^n_D).
\end{equation}
By \eqref{eq:lambda conjugate theorem 3} and \eqref{eq:lambda conjugate theorem 6} we have
\begin{equation}\label{eq:lambda conjugate theorem 10}
m^-(\Lambda_s\oplus\Lambda_e)=\sum_{0<t<T}\dim(\gamma(t)\Lambda_s\cap\Lambda^n_D)+i(\gamma(T))\Lambda_s, \Lambda_e, \Lambda^n_D).
\end{equation} This conclude the proof.
\end{proof}

%{\color{blue}
%\begin{rem} In fact, we also can compute $m^-(\Lambda_s\oplus\Lambda_e)$ by using $\Lambda^n_D\oplus\Lambda_e$. By similar discussions as above we have
%\begin{equation}
%m^-(\Lambda_s\oplus\Lambda_e)=\sum_{0<t<T}\dim(\gamma(t)\Lambda^n_D\cap\Lambda_s)+i(\gamma(T)^{-1}\Lambda_e, \Lambda_s, \Lambda^n_D).
%\end{equation}
%\begin{equation}
%m^-(\Lambda_s\oplus\Lambda_e)=m^-(\Lambda^n_D\oplus\Lambda_e)\mu(\Lambda_e, \gamma(t)\Lambda^n_D)-\dim(\Lambda_e\cap\Lambda_D)+i(\gamma(T)^{-1}\Lambda_e, %\Lambda_s, \Lambda^n_D,-\omega).
%\end{equation}
%But please note that the symplectic structure of $\gamma(T)^{-1}\Lambda_e\cap(\Lambda_s+%\Lambda^n_D)$ is given by $-J_n$.

%\end{rem}   }

\begin{exam}[Index Theorem of \cite{KD}]\label{exam:proof of kalish's result}

Let $(M,g)$ be a complete Riemannian manifold of dimension $d$ and $c:[0,T]\rightarrow M$ be a geodesic. Assume $P$ and $Q$ are two submanifolds of $M$ such that
%\begin{equation}\label{eq:condition of submanifold}
$$
c(0)\in P, \quad \dot{c}(0)\perp P_{c(0)}, \quad c(T)\in Q, \quad \dot{c}(T)\perp Q_{c(T)},
$$
%\end{equation}
where the dimension of $Q_{c(T)}$ is $r$ such that $0\leq r<d$.

Let $\{e_1(0),\ldots,e_{d-1}(0)\}$ be an orthonormal basis of normal bundle $N_{c(0)}M$ at $c(0)$, then by parallel trivialization along $c$, we get an orthonormal basis $\{e_1(t),\ldots,e_{d-1}(t)\}$ of $N_{c(t)}M$. Let $\hat{R}=\hat{R}(\dot{c}, X)\dot{c}$ be the curvature tensor of the Levi-Civita connection. For every Jacobi field along $c$ given by $X(t)=\sum_{i=1}^{d-1}u_i(t)e_i(t)$, the Jacobi equation
%\begin{equation}\label{eq:Jacobi equation}
$$-X''+\hat{R}X=0$$
%\end{equation}
will be reduced to
\begin{equation}\label{eq:reduced Jacobi equation}
-u''(t)+R(t)u(t)=0,
\end{equation}
where $u(t)=(u_1(t),\ldots,u_{d-1}(t))^T$ and $R(t)=[R_{i,j}(t)]_{i,j=1}^{d-1}$ with $R_{i,j}(t)=\langle \hat{R}(\dot{c}(t), e_i(t))\dot{c}(t), e_j(t)\rangle$.

Let $\hat{S}(t)$ be the second fundamental form of $P$ at $c(t)$ with respect to $\dot{c}(t)$. Denote $S(t)=[S_{i,j}(t)]_{i,j=1}^{d-1}$ with $S_{i,j}(t)=\langle \hat{S}(t)e_i(t), e_j(t)\rangle$. Then $R(t)$ and $S(t)$ are both symmetric. The index form
%\begin{equation}\label{eq:index form in riemannian manifold}
$$
\hat{I}(X,Y)=\int_0^T\langle \hat{R}X-X'', Y\rangle dt+\langle X'-\hat{S}X,Y\rangle|_0^T
$$
%\end{equation}
is converted into
%\begin{equation}\label{eq:reduced index form in riemannian manifold}
$$I(u,v)=\int_0^T\langle Ru-u'', v\rangle dt+\langle u'-Su,v\rangle|_0^T.$$
%\end{equation}
Since $\hat{I}$ is defined on $H=\{V(t)=\sum_{i=1}^{d-1}u_i(t)e_i(t), V(0)\in P_{c(0)}, V(T)\in Q_{c(T)}\}$, namely, the linear space of continuous piecewise $C^\infty$ vector fields along $\gamma$ which are orthogonal to $\gamma$ and whose initial and final vectors are in $P_{c(0)}$ and $Q_{c(T)}$ respectively, then $I$ is defined on $\mathcal{H}=\{u\in W^{1,2}([0,T],\mathbf{R}^{d-1}), u(0)\in V_0, u(T)\in V_T\}$,    where $V_0, V_T\subset\mathbf{R}^{d-1}$  can be considered as the tangent space of $P$, $Q$ at $c(0)$, $c(T)$ separately. Note that the index form $\hat{I}$ is a little different from the original one in \cite{KD} since we have assumed that $X'$ is continuous, this is not an essential problem.

A Jacobi field $X$ is called $P$-Jacobi field if it is orthogonal to $c$ such that $X(0)\in P_{c(0)}$ and $X'(0)-S_0X(0)\perp P_{c(0)}$. If $X(t)=\sum_{i=1}^{d-1}u_i(t)e_i(t)$ is a $P$-Jacobi field then by direct computations we have $u(t)$ is a solution of system \eqref{eq:reduced Jacobi equation} such that $u(0)\in V_0$ and $u'(0)-S(0)u(0)\in V_0^{\bot}$. A $P$-focal point is a point $c(t_0), t_0\in(0,T]$ such that there exists a nonzero $P$-Jacobi field which vanishes at $t_0$.

Let $y(t)=u'(t)$ and $z(t)=(y(t), x(t))^T$, then system \eqref{eq:reduced Jacobi equation} is converted into Hamiltonian system \eqref{eq:Hamilton system} with
 $\mathcal{B}(t)=\begin{bmatrix}I_n&0\\0&-R\end{bmatrix}$ and the boundary condition is given by  $$z(0)\in\Lambda_s:=\{y(0)-S(0)x(0)\in J_{d-1}V^\perp_0\}, \,\ z(T)\in\Lambda_e:=\{y(T)-S(T)x(T)\in J_{d-1}V^\perp_T\}.$$  Let $\gamma(t)$ be the fundamental solution.
 % of system \eqref{eq:hamiltonian system from riemannian manifold}.
 In fact, after parallel trivialization, a $P$-focal point is equivalent to $\Lambda_s$ conjugate point which is defined above Theorem \ref{thm: morse index thm gene}, then by Theorem \ref{thm: morse index thm gene} we have
\begin{equation}
m^-(\Lambda_s\oplus\Lambda_e)=\sum_{0<t_0<T} \dim(\gamma(t_0)\Lambda_s\cap\Lambda^n_D)+ i(\gamma(T)\Lambda_s,\Lambda_e,\Lambda^n_D).
\end{equation}
Note that $\dim(\gamma(t_0)\Lambda_s\cap\Lambda^n_D)$ is just the multiplicity of $P$-focal point $c(t_0)$ and $i(\gamma(T)\Lambda_s,\Lambda_e,\Lambda^n_D)$ is just the difference $m^-(\Lambda_s\oplus\Lambda_e)-m^-(\Lambda_s\oplus\Lambda^n_D)$ which is equal to $i(A)$ in \cite{KD} under the non-degenerate assumption.  We conclude this example.

\end{exam}

\section{The stability of brake orbits}

Recall that a brake orbit $x(t)$ of system \eqref{eq:S-L equation} satisfies that
\begin{equation}\label{eq:brake of S-L system}
\dot{x}(0)=\dot{x}(T/2)=0, \quad x(T/2+t)=x(T/2-t), \quad x(T+t)=x(t), \quad \forall t\in \mathbf{R}.
\end{equation}
Let $\mathcal{N}=\begin{bmatrix}-I_n&0\\0&I_n\end{bmatrix}$, then the corresponding brake orbit $z(t)$ of system \eqref{eq:Hamilton system} satisfies that
\begin{equation}\label{eq:brake symmetry}
z(-t)=\mathcal{N}z(t), \quad z(t+T)=z(t), \quad \forall t\in \mathbf{R}.
\end{equation}
It's easy to check that
\begin{equation}\label{eq:condition of monodromy matrix of brake orbit}
(\mathcal{N}\gamma(T))^2=I_{2n}, \quad \gamma(T)=\mathcal{N}\gamma(\frac{T}{2})^{-1}\mathcal{N}\gamma(\frac{T}{2}).
\end{equation}

Before to give another corollary of Theorem \ref{thm:difference of two morse index} and the application to brake orbits, we need an important but simple lemma.
\begin{lem}\label{lem:principle of istropic subspace}
For a given linear subspace $\mathcal{W}\subset\mathbf{R}^n$, let $\mathfrak{Q}$ be a quadratic form defined on $\mathcal{W}$ and $\mathcal{X}\subset\mathcal{W}$ be any subspace such that
$\mathfrak{Q}(x,y)=0$ for every $x,y\in\mathcal{X}$.  Then we have
\begin{equation}\label{eq:dimension of Q-istropic}
\dim\mathcal{X}\leq m^+(\mathfrak{Q})+\dim\ker \mathfrak{Q}.
\end{equation}
Particularly, if we have the splitting $\ker \mathfrak{Q}=\ker_1\mathfrak{Q}\oplus\ker_2\mathfrak{Q}$ and $\mathcal{X}\cap\ker_1\mathfrak{Q}=\{0\}$, then for $\mathcal{X}$ as above there holds
\begin{equation}\label{eq:dimension of Q-istropic with splitting}
\dim\mathcal{X}\leq m^+(\mathfrak{Q})+\dim\ker_2\mathfrak{Q}.
\end{equation}
\end{lem}
\begin{proof}
Associated to $\mathfrak{Q}$, there exists a symmetric matrix $\mathcal{A}$ such that $\mathfrak{Q}(x,y)=\langle\mathcal{A}x,y\rangle$, then we have the splitting
%\begin{equation}
$$\mathcal{W}=\mathcal{W}^+\oplus\mathcal{W}^0\oplus\mathcal{W}^-,$$
%\end{equation}
where $\mathcal{W}^*, *=+,0,-$ denote the corresponding positive, zero and negative eigenspaces of $\mathcal{A}$. If $\dim\mathcal{X}> m^+(\mathfrak{Q})+\dim\ker \mathfrak{Q}$, then $\mathcal{X}\cap\mathcal{W}^-\neq\{0\}$ which is a contradiction to $\mathfrak{Q}(x,y)=0$ for every $x,y\in\mathcal{X}$.

For the second statement, if \eqref{eq:dimension of Q-istropic with splitting} is false, then $\mathcal{X}\cap(\ker_1\mathfrak{Q}\oplus\mathcal{W}^-)\neq\{0\}$. There must exist a nonzero $x=x^1_0+x^-\in \mathcal{X}$ with $x^1_0\in\ker_1\mathfrak{Q}, x^-\in\mathcal{W}$ such that $\mathfrak{Q}(x,x)=\mathfrak{Q}(x^1_0,x^1_0)+\mathfrak{Q}(x^-,x^-)=0$. It derives $\mathfrak{Q}(x^-,x^-)=0$ and consequently $x^-=0$. Hence $x=x^1_0\in\ker_1\mathfrak{Q}\cap\mathcal{X}$ which is to contradiction to $\mathcal{X}\cap\ker_1\mathfrak{Q}=\{0\}$. We complete the proof.
\end{proof}

Now we can prove another corollary of Theorem \ref{thm:difference of two morse index} as following:
\begin{cor}\label{cor:dimension of an eigenvalue}
Let $\gamma(T)$ be the monodromy matrix of system \eqref{eq:Hamilton system} and $\mathbf{U}$ be the unit circle of complex plane, then
\begin{equation}
\dim\ker(\gamma(T)-\lambda I_{2n})\leq m^-(\Lambda_N)+m^0(\Lambda_N)-m^-(\Lambda_D).
\end{equation}
 %If there holds $Gr(\gamma(T))\cap\Lambda_N=\{0\}$, then for every $\lambda\in\sigma(\gamma(T))\cap\mathbf{U}$ we have
%\begin{equation}
%\dim\ker(\gamma(T)-\lambda I_{2n})\leq m^-(\Lambda_N)-m^-(\Lambda_D).
%\end{equation}
\end{cor}
\begin{proof}
Let $\gamma(T)=\begin{bmatrix}D_1&D_2\\D_3&D_4\end{bmatrix}$, please recall that the frame of $Gr(\gamma)$ under the symplectic form $J_{2n}$ is given by $\begin{bmatrix} -I&0\\D_1&D_2\\0&I\\D_3&D_4 \end{bmatrix}$.  For every $\lambda\in\sigma(\gamma(T))$ with corresponding eigenvector $u_{\lambda}=(x,y)^T$, namely, $\gamma(T)u_{\lambda}=\lambda u_{\lambda}$,
 let  $$z_\lambda=(-x, D_1x+D_2y, y,D_3x+D_4y)^T\in Gr(\gamma(T))\cap(\Lambda_N+\Lambda_D).$$ By the definition of $\mathfrak{Q}(Gr(\gamma(T)),\Lambda_N;\Lambda_D)$ and direct computations we have
\begin{equation}\label{eq:form of Q on graph of monodromy matrix}
\mathfrak{Q}(z_\lambda,z_\lambda)=\langle x,y\rangle-\langle D_1x+D_2y,D_3x+D_4y\rangle.
\end{equation}
For every $\lambda\in\sigma(\gamma(T))\cap\mathbf{U}$ and every $(x,y)^T\in\ker(\gamma(T)-\lambda I_{2n})$ which derives $ D_1x+D_2y=\lambda x, D_3x+D_4y=\lambda y$, then by \eqref{eq:form of Q on graph of monodromy matrix} we have $\mathfrak{Q}(z_\lambda,z_\lambda)=0$.   By Lemma \ref{lem:principle of istropic subspace} we get
 %\begin{equation}\label{eq:inequality of eigenvalue on U}
 $$\dim\ker(\gamma(T)-\lambda I_{2n})\leq m^+(\mathfrak{Q})+\dim\ker \mathfrak{Q}.$$
% \end{equation}
  But note that $$\ker \mathfrak{Q}=(Gr(\gamma(T))\cap\Lambda_D)\oplus(Gr(\gamma(T))\cap\Lambda_N),$$
  then \bea \dim\ker(\gamma(T)-\lambda I_{2n})\leq i(Gr(\gamma(T)), \Lambda_N, \Lambda_D)+\dim Gr(\gamma(T))\cap \Lambda_N. \label{4.10}  \eea
 From the fact that $m^0(\Lambda_N)=\dim Gr(\gamma(T))\cap \Lambda_N$ and combine with Theorem \ref{thm:difference of two morse index} and \eqref{4.10}, we complete the proof.
  \end{proof}

Now we will go on the application to a brake orbit. We still denote $\gamma(T)=\begin{bmatrix}D_1&D_2\\D_3&D_4\end{bmatrix}$ and $\gamma(T/2)=\begin{bmatrix}E_1&E_2\\E_3&E_4\end{bmatrix}$. Since $\gamma(T/2)\in\Sp(2n)$, then
$$ E_1^TE_3=E^T_3E_1,\quad   E^T_2E_4=E_4^TE_2,\quad  E^T_4E_1-E^T_2E_3=I_n. $$
From  $\gamma(T)=\mathcal{N}\gamma^{-1}(T/2)\mathcal{N}\gamma(T/2)\in\Sp(2n)$, direct computations show that
\bea  D_1=E_4^TE_1+E^T_2E_3, \quad D_2=E_4^TE_2+E^T_2E_4,\quad D_3=E_3^TE_1+E^T_1E_3,\quad D_4=E_3^TE_2+E^T_1E_4, \label{M2decom}  \eea
and we have
\bea D_4=D^T_1,\quad D_2=D_2^T, \quad  D_3=D_3^T,\quad  D^T_1D_3=D_3D_1, \quad D_2D^T_1=D_1D_2,  \quad D^2_1-D_2D_3=I_n. \label{eq:property of gamma(T)} \eea

We start from a simple lemma.
\begin{lem}\label{cor:fundamentle condition of Q-orthogonal}
For $\lambda_1, \lambda_2\in\sigma(\gamma(T))$ with corresponding eigenvectors $u_1=(x_1,y_1)^T, \ u_2=(x_2,y_2)^T$, then there must hold at least one of the three statements: $(1)\ \lambda_1=\bar{\lambda}_2, \quad (2)\ \lambda_1\bar{\lambda}_2=1 \quad (3)\ \langle x_1,y_2\rangle=0$.
\end{lem}
\begin{proof}
%Note that $\gamma(T)$ is symplectic, then $\gamma(T)^TJ\gamma(T)=J$, combining with \eqref{eq:condition of monodromy matrix of brake orbit} we can easily compute that
%\begin{equation}\label{eq:property of gamma(T)}
%D_1=D_4^*, \quad D_2=D_2^*, \quad D_3=D_3^*.
%\end{equation}
For every $\lambda\in\sigma(\gamma(T))$ with corresponding eigenvector $u_{\lambda}=(x,y)^T$, namely, $\gamma(T)u_{\lambda}=\lambda u_{\lambda}$, by \eqref{eq:condition of monodromy matrix of brake orbit} we have $\gamma(T)\mathcal{N}u_{\lambda}=\mathcal{N}\gamma(T)^{-1}u_{\lambda}=\lambda^{-1}\mathcal{N}u_{\lambda}$, then $\mathcal{N}u_{\lambda}=(-x,y)^T$ is an eigenvector related to $\lambda^{-1}\in\sigma(\gamma(T))$. By the expression of $\gamma(T)$ and some direct calculations, there hold
\begin{equation}\label{eq:conditions of eigenvector of gamma(T)}
D_1x=\frac{\lambda+\lambda^{-1}}{2}x, \ D_2y=\frac{\lambda-\lambda^{-1}}{2}y, \ D_3x=\frac{\lambda-\lambda^{-1}}{2}y, \ D_4y=\frac{\lambda+\lambda^{-1}}{2}y.
\end{equation}
For $\lambda_1, \lambda_2\in\sigma(\gamma(T))$ with corresponding eigenvectors $u_1=(x_1,y_1)^T, u_2=(x_2,y_2)^T$, then by \eqref{eq:property of gamma(T)} and \eqref{eq:conditions of eigenvector of gamma(T)} we have
%\begin{equation}
$$\langle \frac{\lambda_1+\lambda_1^{-1}}{2}x_1,y_2 \rangle=\langle D_1x_1,y_2\rangle=\langle x_1,D_1^*y_2 \rangle=\langle x_1,D_4y_2\rangle=
\langle x_1, \frac{\lambda_2+\lambda_2^{-1}}{2}y_2\rangle$$
%\end{equation}
which derives
\begin{equation}\label{eq:relation about two eigenspace}
(\frac{\lambda_1+\lambda_1^{-1}}{2}-\frac{\bar{\lambda}_2+\bar{\lambda}_2^{-1}}{2})\langle x_1,y_2 \rangle=0.
\end{equation}
In fact, $\frac{\lambda_1+\lambda_1^{-1}}{2}-\frac{\bar{\lambda}_2+\bar{\lambda}_2^{-1}}{2}=0$ is equivalent to $\lambda_1=\bar{\lambda}_2$ or $\lambda_1\bar{\lambda}_2=1$. By \eqref{eq:relation about two eigenspace} we complete the proof.
\end{proof}

\begin{cor}\label{cor4.4}
For $i=1,2$,   $\lambda_i\in\sigma(\gamma(T))$ with eigenvector $u_i$,  let  $z_i=(-x_i,D_1x_i+D_2y_i,y_i,D_3x_i+D_4y_i)^T$,  then as long as $\lambda_1\neq\bar{\lambda}_2$ or $\lambda_1=\lambda_2=\pm1$,  we have $z_1$ and $z_2$ are $\mathfrak{Q}$-orthogonal.
\end{cor}
\begin{proof}
Consider the form $\mathfrak{Q}(Gr(\gamma(T)), \Lambda_N; \Lambda_D)$ defined on $Gr(\gamma(T))\cap(\Lambda_N+\Lambda_D)$ and let $\lambda_i, u_i, i=1,2 $ be as above, then by \eqref{eq:form of Q on graph of monodromy matrix} and direct computations we have
\begin{equation}\label{eq:concrete formula of form Q}
\mathfrak{Q}(z_1,z_2)=(1-\lambda_1\bar{\lambda}_2)\langle x_1,y_2\rangle.
\end{equation}
Obviously, if $1=\lambda_1\bar{\lambda}_2$, then $z_1$ and $z_2$ are $\mathfrak{Q}$-orthogonal. If $1\neq\lambda_1\bar{\lambda}_2$ and $\lambda_1\neq\bar{\lambda}_2$, by Lemma \ref{cor:fundamentle condition of Q-orthogonal}, we still have $z_1$ and $z_2$ are $\mathfrak{Q}$-orthogonal.
\end{proof}

%Another lemma we need will play an essential role in the proof of Theorem \ref{thm:stability of brake orbit}. It can be stated as
%\begin{lem}\label{lem:estimate of eigenvalues of brake orbit}
%Let $\mathbf{C}^\perp$ be as in Theorem \ref{thm:stability of brake orbit}, then for the brake orbit we have the following estimation
%\begin{equation}\label{eq:estimate dimension by difference of morse index1}
%\dim\bigoplus_{\lambda\in\sigma(\gamma(T))\cap(\mathbf{C}^\perp\cup\{\pm1\})}\ker(\gamma(T)-\lambda I_{2n})\leq m^-(\Lambda_N)+m^0(\Lambda_N)-m^-(\Lambda_D).
%\end{equation}
%and
%\begin{equation}\label{eq:estimate dimension by difference of morse index}
%\dim\bigoplus_{\lambda\in\sigma(\gamma(T))\cap\mathbf{C}^\perp}\ker(\gamma(T)-\lambda I_{2n})\leq m^-(\Lambda_N)-m^-(\Lambda_D).
%\end{equation}
%\end{lem}
Now we give the proof of Theorem \ref{thm:stability of brake orbit1}.
\begin{proof}[Proof of Theorem \ref{thm:stability of brake orbit1}]

For every $\lambda\in\sigma(\gamma(T))$ with corresponding eigenvector $u_{\lambda}=(x,y)^T$, we denote $z=(-x,D_1x+D_2y, y, D_3x+D_4y)^T$. In fact, it's apparent that $\lambda_1\neq\bar{\lambda}_2$ for any $\lambda_1, \lambda_2\in \sigma(\gamma(T))\cap\mathbf{C}^\perp\cup\{\pm1\}$, then by Corollary  \ref{cor4.4},
 we have $z_1, z_2$ are $\mathfrak{Q}$-orthogonal.  We get  \eqref{eq:estimate of eigenvalues not real1}   from Lemma \ref{lem:principle of istropic subspace}

 To prove \eqref{eq:estimate of eigenvalues not real},
 recall that $\ker \mathfrak{Q}(Gr(\gamma(T)), \Lambda_N; \Lambda_D)=\ker_1\mathfrak{Q}\oplus\ker_2\mathfrak{Q}$ where $\ker_1\mathfrak{Q}=Gr(\gamma(T))\cap\Lambda_N$ and $\ker_2\mathfrak{Q}=Gr(\gamma(T))\cap\Lambda_D$. For every $\lambda\in\sigma(\gamma(T))\cap\mathbf{C}^\perp$, if $z=(-x,D_1x+D_2y, y, D_3x+D_4y)^T\in\Lambda_N$, then $x=0$. Together with \eqref{eq:conditions of eigenvector of gamma(T)} which derives $D_3x=\frac{\lambda-\lambda^{-1}}{2}y=0$, then $y$=0 since $\frac{\lambda-\lambda^{-1}}{2}\neq0$. So $z\notin\Lambda_N$. Note that for $\lambda_i, \lambda_j\in \sigma(\gamma(T))\cap\mathbf{C}^\perp$, there must hold $\frac{\lambda_i+\lambda_i^{-1}}{2}\neq\frac{\lambda_j+\lambda_j^{-1}}{2}$ if $\lambda_i\neq\lambda_j$. By \eqref{eq:conditions of eigenvector of gamma(T)} we know $D_1x_i=\frac{\lambda_i+\lambda_i^{-1}}{2}x_i$ which deduces that all $x_i$ are linear independent for different $\lambda_i$. If there is a linear combination $\sum_ia_iz_i\in\Lambda_N$, then we have $\sum_ia_ix_i=0$ and consequently every $a_i=0$. Then every $z=\sum_ia_iz_i$ is not in $\Lambda_N$ which means $\ker \mathfrak{Q}\cap Gr(\gamma(T))\cap\Lambda_N=\{0\}$.  Let $\mathcal{X}$ be the space of all $z=(-x,D_1x+D_2y, y, D_3x+D_4y)^T$ such that $(x,y)^T$ is an eigenvector of some $\lambda\in\sigma(\gamma(T))\cap\mathbf{C}^\perp$. Then $\mathfrak{Q}|_{\mathcal{X}}=0$ and $\mathcal{X}\cap \ker_1\mathfrak{Q}=\{0\}$. By \eqref{eq:dimension of Q-istropic with splitting} there hold
\begin{equation}\label{eq:dimension less than difference of morse index}
\dim\bigoplus_{\lambda\in\sigma(\gamma(T))\cap\mathbf{C}^\perp}\ker(\gamma(T)-\lambda I_{2n})\leq m^+(\mathfrak{Q})+\dim(Gr(\gamma(T))\cap\Lambda_D).
\end{equation}
Recall the formula \eqref{eq:compute the triple index} and Theorem \ref{thm:difference of two morse index}, the righthand side of \eqref{eq:dimension less than difference of morse index} is exactly $m^-(\Lambda_N)-m^-(\Lambda_D)$. We complete the proof.
\end{proof}
Now we can present the proof of Theorem \ref{thm:stability of brake orbit}.
\begin{proof}[Proof of Theorem \ref{thm:stability of brake orbit}]

Let $m^-([a,b];\Lambda_0)$ be the Morse index on time interval $[a,b]$ for the boundary condition $\Lambda_0$. Denote $V_{\pm}(\mathcal{N})$ the eigenspace corresponding to the eigenvalue $\pm1$ of $\mathcal{N}$. By simple computations we have $V_+(\mathcal{N})=\Lambda^n_N$ and $V_-(\mathcal{N})=\Lambda^n_D$. Then by \cite[Theorem 2]{HPY} we have
\begin{equation}\label{eq:morse index decomposition}
\begin{aligned}
k=m^-(x,\Lambda_P)&=m^-([0,T/2];\Lambda_N)+m^-([0,T/2];\Lambda_D)\\
m^-([0,T];\Lambda_N)&=m^-([0,T/2];\Lambda_N)+m^-([0,T/2];\Lambda^n_N\oplus\Lambda^n_D)\\
m^-([0,T];\Lambda_D)&=m^-([0,T/2];\Lambda_D)+m^-([0,T/2];\Lambda^n_D\oplus\Lambda^n_N)
\end{aligned}
\end{equation}
which obviously derives
\begin{equation}
m^-(\Lambda_N)-m^-(\Lambda_D)\leq2k.
\end{equation}
By \eqref{eq:estimate of eigenvalues not real} we complete the proof.

\end{proof}

As a corollary, we give a new proof for  Ure\~{n}a's \cite[Theorem 1.1]{Ur}  interesting result.
\begin{thm}
For a minimizer brake orbit all eigenvalues of monodromy matrix $\gamma(T)$ are real and positive.\end{thm}
\begin{proof}
Recently, in \cite[Theorem 1.1]{Ur} the author proves an great result that for a minimizer brake orbit all eigenvalues of monodromy matrix $\gamma(T)$ are real and positive. In fact, this result can be easily proved by Theorem \ref{thm:difference of two morse index} and Theorem \ref{thm:stability of brake orbit}. Precisely, since the brake orbit is a minimizer, then the Morse index $k=0$ in Theorem \ref{thm:stability of brake orbit} which means all eigenvalues of $\gamma(T)$ are real. Moreover, $m^-(\Lambda_N)=m^-(\Lambda_D)=0$. The rest work is only to prove all $\lambda\in\sigma(\gamma(T))$ are positive. By theorem \ref{thm:difference of two morse index} and \eqref{eq:compute the triple index}, we have
\begin{equation}
\mathfrak{Q}(Gr(\gamma(T)),\Lambda_N;\Lambda_D)\leq0, \quad Gr(\gamma(T))\cap\Lambda_D=\{0\}.
\end{equation}
In fact, $Gr(\gamma(T))\cap\Lambda_D=\{0\}$ can derive $D_3$ is invertible. By \eqref{eq:form of Q on graph of monodromy matrix} and $\mathfrak{Q}(Gr(\gamma(T)),\Lambda_N;\Lambda_D)\leq0$ we have $D_3D_1\geq0$. By \eqref{eq:condition of monodromy matrix of brake orbit} and \eqref{M2decom},
we have $D_3=2E_3^TE_1$. By \eqref{eq:morse index decomposition}, we have $m^-([0,T/2];\Lambda_N)=m^-([0,T/2];\Lambda_D)=0$.
 By Theorem \ref{thm:difference of two morse index} again, we have $E_3^TE_1\geq0$. Consequently, $D_3\geq0$. Recall that $D_3$ is invertible, then $D_3>0$. Since $D_3^{-1/2}(D_3D_1)D_3^{-1/2}$ is similar to $D_1$, then $D_1\geq0$. By \eqref{eq:conditions of eigenvector of gamma(T)}, there hold for every $\lambda\in\sigma(\gamma(T))$ is equivalent to $\frac{\lambda+\lambda^{-1}}{2}\in\sigma(D_1)$, then there must hold all  $\lambda\in\sigma(\gamma(T))$ are positive.
 \end{proof}

\noindent {\bf Acknowledgements.}  The authors sincerely thank Professor Alessandro Portaluri for the discussions of index theory and the stability problem of brake orbits.

\end{document}